\documentclass{amsart}

\title[H\"older continuity]{H\"older continuity of core entropy for non-recurrent quadratic polynomials}
\author{Malte Ha{\ss}ler}
\address{Cornell University, Department of Mathematics, Malott Hall, 212 Garden Avenue, 14853 Ithaca, USA}
\email{mh2479@cornell.edu}
\author{Dierk Schleicher}
\address{Aix-Marseille Universit\'e and CNRS, UMR 7373, Institut de Math\'ematiques de Marseille, 163 Avenue de Luminy Case 901, 13009 Marseille, France. 
}
\email{dierk.schleicher@univ-amu.fr}
\thanks{This work was funded through the ERC grant 695621 \textit{HOLOGRAM}}

\usepackage{amsthm}
\usepackage{graphicx}

\usepackage{amsmath, amsthm, amsfonts,amssymb, comment}

\theoremstyle{plain}

\newtheorem{theorem}{Theorem}

\newtheorem{lemma}[theorem]{Lemma}
\newtheorem{proposition}[theorem]{Proposition}
\newtheorem{corollary}[theorem]{Corollary} 
\newtheorem*{conjecture}{Conjecture} 

\newtheorem*{MainTheorem}{Main Theorem}

\theoremstyle{definition}
\newtheorem{definition}[theorem]{Definition}
\newtheorem*{remark}{Remark}

\newcounter{reminder}

\newcommand{\Hide}[1]{}
\newcommand{\hide}[1]{}

\renewcommand{\theta}{\vartheta}
\renewcommand{\phi}{\varphi}
\newcommand{\0}{\mathtt{0}}
\newcommand{\1}{\mathtt{1}}
\newcommand{\K}{\mathbb K}
\newcommand{\e}{e}
\newcommand{\up}[1]{{{#1}_+}}
\newcommand{\low}[1]{{{#1}_-}}
\newcommand{\sym}{\{\0,\1\}}
\newcommand{\lo}{\low}
\newcommand{\orb}{\text{orb}}
\newcommand{\Newpage}

\newcommand{\eps}{\varepsilon}
\newcommand{\C}{\mathbb C}
\newcommand{\Z}{\mathbb Z}
\newcommand{\R}{\mathbb R}

\newcommand{\N}{{\mathbb N}}
\newcommand{\M}{\mathcal{M}}
\newcommand{\Hub}{\mathcal{H}}
\newcommand{\sm}{\setminus}

\newcommand{\ovl}[1]{\overline{#1}}
\newcommand{\Circle}{\mathbb S^1}
\newcommand{\Cbar}{\ovl \C}
\newcommand{\disk}{\mathbb D}
\newcommand{\diskbar}{\ovl\disk}

\newcommand{\LM}{\mathcal{LM}}

\numberwithin{theorem}{section}
\numberwithin{equation}{section}

\newcommand{\diff}{\operatorname{diff}}
\newcommand{\Diff}{\operatorname{Diff}}

\begin{document}

\maketitle

\begin{abstract}
We prove that core entropy is H\"older continuous as a function of external angles for a large class of quadratic polynomials that are non-recurrent with respect to angle-doubling, in particular all of them that exhibit a finite Hubbard tree. The result follows from a symbolic analysis of the Mandelbrot set and the dynamics of Hubbard trees in terms of kneading sequences which has been established in previous work. 
\end{abstract}
\medskip
\small \textbf{Keywords:} \textit{Complex and symbolic dynamics,  topological entropy, renormalization}

\section{Introduction}

Entropy (more precisely, topological entropy) is a quantitative measure of the complexity of dynamical systems, essentially describing the exponential growth rate of the number of ``substantially different'' finite orbits of length $n$. It is closely related to concepts in physics and information theory, where it also measures the exponential growth rate of the number of configurations.

Specifically in real and complex dynamics, topological entropy has been investigated for a variety of classes of dynamical systems. In complex dynamics, Lyubich~\cite{LyubichEntropyRussian, LyubichEntropy} proved that topological entropy for rational maps of degree $d$ always equals $\log d$, so it only depends on the degree. In contrast, in real dynamics topological entropy is often considered for the restriction to a ``dynamical interval'', an invariant interval that contains the interesting dynamics. In this context, topological entropy is not constant, and questions such as monotonicity and continuity with respect to a parameter are of interest.

Thurston introduced the concept of \emph{core entropy} for (postcritically finite) polynomials as the topological entropy of the restriction to the Hubbard tree, an invariant subset of the dynamical plane that, like the dynamical interval for real maps, captures the interesting dynamics. He raised the issue whether core entropy can be extended to polynomials beyond postcritically finite ones, and whether such an extension could be continuous. 

The first extension to general polynomials had been investigated under the name of \emph{biaccessibility dimension} in \cite{HenkBiaccDim,PhilippDierk} even before Thurston raised the question and discovered the connection. Continuity of core entropy has been proved for quadratic polynomials independently in \cite{DimaEntropy} and \cite{TiozzoContinuity}, and for general polynomials in \cite{GaoTiozzo}.

Specifically for quadratic polynomials, the following conjecture has been raised quite a while ago:

\begin{conjecture}[H\"older continuity of core entropy]
The continuous dependence of core entropy on the polynomial is H\"older (with respect to an appropriate measure on the space of polynomials), such that the H\"older exponent is equal (up to a scaling constant) to core entropy itself; in particular, H\"older continuity fails when core entropy is zero.
\end{conjecture}

The conjecture is natural in the sense that as early as 1980 it was observed by Guckenheimer \cite{Guckenheimer} that in the context of real unimodal interval maps, entropy is H\"older with respect to a parameter when the entropy is positive.  

Several proofs of H\"older continuity of core entropy have been announced a number of years ago (in particular, by Bruin in 2012, and by Fels in 2016), but never substantiated. A partial result for real quadratic polynomials has been proven by Tiozzo \cite{TiozzoHoelderReal}.

In this paper we prove this conjecture for a large class of complex quadratic polynomials, but not all; in particular, we assume non-recurrence of the polynomial.

We pursue a purely combinatorial approach. The combinatorics of our polynomials is specified in terms of an external angle, or in terms of a kneading sequence. The combinatorial notions are explained below; first we state our main result.

\begin{MainTheorem}[H\"older continuity for angles]
\label{Thm:HoelderAngles}
    Let $\theta$ be an external angle that is non-recurrent with respect to angle doubling and has a finite Hubbard tree. Then core entropy in terms of external angles is H\"older continuous locally at $\theta$ with H\"older exponent $(h(\theta)-\eps)/\log 2$ for every $\eps \in(0,h(\nu))$.
    
Furthermore, the same holds for non-recurrent external angles with infinite trees provided that the associated polynomial is renormalizable with a  base sequence that has positive core entropy.
\end{MainTheorem}

Note that the phrasing of the theorem is based on the statement, discussed below, that \emph{every} quadratic polynomial with connected Julia set, postcritically finite or not, has an associated Hubbard tree, which may be an infinite tree: this is the main result of \cite{MalteDierkTrees}. 

The condition for infinite trees can be rephrased as saying that the renormalizable polynomial is contained in a little Mandelbrot set for which the main hyperbolic component is not contained in the ``main molecule'' of the Mandelbrot set. 

Most of our estimates are concerned with establishing H\"older continuity with the conjectured exponent $h(\theta)$ (up to arbitrarily small $\eps$). Since core entropy is at most $\log 2$, the H\"older exponent is at most 1. The conjecture also predicts that this exponent is best possible; we show this in Theorem~\ref{Thm:larger_expo} for a dense subset of kneading sequences.

It is well known that core entropy is in general not H\"older when the core entropy vanishes; our restriction to positive core entropy is thus not a loss. We describe a simple example in Corollary \ref{Cor:zero_entropy}. 

In Section~\ref{Sec:Htrees} we present our combinatorial construction of abstract Hubbard trees as established in \cite{MalteDierkTrees}. In Section~\ref{Sec:parameter} we develop a natural partial order for the space of all kneading sequences (i.e.\ parameter space). 

We define core entropy of a kneading sequence in Section~\ref{Sec:entropy} and show its basic properties. In Section~\ref{Sec:estimates} we introduce special classes of kneading sequences, in particular renormalizable ones. H\"{o}lder continuity of core entropy as a function of kneading sequences is proven in Sections~\ref{Sec:hoelder_nr} and \ref{Sec:hoelder_r}, depending on whether or not the kneading sequence is renormalizable. Finally, we translate these results in terms of external angles in Section~\ref{Sec:angles}.  

We should emphasize that in our setting, core entropy is defined naturally in terms of kneading sequences,  and most of our work concerns this setting.
Kneading sequences have been a key tool for the study of real quadratic polynomials since they have been introduced by Milnor and Thurston in 1979 \cite{MilnorThurston}, and they have a natural extension to complex quadratic polynomials (see for instance \cite{Spiders}). In particular, every external angle (that describes a complex quadratic polynomial) has an associated kneading sequence, but there are many kneading sequences that are not coming from external angles or quadratic polynomials: so-called ``not complex admissible'' kneading sequences; see \cite{DierkHenkAdmiss}. Core entropy is still defined for such kneading sequences (see Definition~\ref{Def:KneadingEntropy}), and our major  results (see Theorems~\ref{Thm:HoelderNonRenorm}, \ref{Thm:HoelderRenorm} and \ref{Thm:Hoelder_low}) concern H\"older continuity of core entropy as a function of kneading sequences (under the same certain conditions such as non-recurrence).

\newpage

\section{Kneading sequences and Hubbard trees}
\label{Sec:Htrees}

The central object of our discussion are \emph{kneading sequences}. Traditionally, these are defined for unimodal real maps \cite{MilnorThurston}. There is a well known extension of kneading sequences to complex postcritically finite quadratic polynomials: each of these has a finite invariant Hubbard tree such that the unique critical point divides the tree into two components. The kneading sequence is the infinite sequence over the alphabet $\{\0,\1\}$ describing the symbolic dynamics of the critical orbit in the Hubbard tree with respect to the unique critical point. In the special case that the critical point is periodic, a $\star$ indicates the positions where the critical orbit returns to the critical point. 

Our approach is more abstract and more general and goes in the opposite direction: we start with an abstract kneading sequence and construct from it a dynamical system and a Hubbard tree, and in particular extract  core entropy. Of course, in traditional cases where a Hubbard tree is defined, our tree coindices with the traditional tree (a Hubbard tree for a given kneading sequence is unique subject to certain properties). 

\begin{definition}[Kneading sequence]
A \emph{$\star$-periodic kneading sequence} is an infinite periodic sequence over the alphabet $\{\0,\1,\star\}$ where a $\star$ occurs exactly once within the period, at the last position. 

A \emph{kneading sequence} is either a non-periodic infinite sequence over the alphabet $\{\0,\1\}$, or a $\star$-periodic kneading sequence.

The unique $\star$-periodic kneading sequence of period $1$ is called the \emph{trivial kneading sequence} $\ovl\star=\star\star\star\dots$. Here and elsewhere, the overbar denotes a periodic repetition, such as $\1\ovl{\0\1}=\1\,\0\1\,\0\1\,\0\1\dots$

By convention, every non-trivial kneading sequence starts with the symbol $\1$.
\end{definition}

\begin{remark}
Kneading sequences occur naturally for the dynamics of angle doubling on the circle; see for instance \cite{HenkDierkPreprint}. In this context, periodic kneading sequences without $\star$ occur naturally, but for our purposes they are not relevant. On the other hand, not all kneading sequences occur by angle doubling: those that do are called {``complex admissible''}; they are classified in \cite{DierkHenkAdmiss}. The more general kneading sequences with respect to our definition are often called ``abstract kneading sequences''.
\end{remark}

\begin{definition}[Dynamical system associated to kneading sequence]
Every kneading sequence $\nu$ ($\star$-periodic or not) has an associated dynamical system $(X_\nu,\sigma)$ where $X_\nu$ is a collection of infinite sequences over the alphabet $\{\0,\1,\star\}$ and the dynamics is given by $\sigma$, the (left) shift on sequences.  
\end{definition}

The collection $X_\nu$ contains the sequence $\nu$, called the \emph{critical value}, and its $\sigma$-preimage $\star\nu$ called the \emph{critical point}. Moreover, it contains all \emph{postcritical points} $\nu^k:=\sigma^k(\nu)$ for $k\ge 0$, as well as \emph{precritical points} of the form $w\star\nu$, where $w$ is a finite word over $\{\0,\1\}$ (possibly empty). The number of iterations until a precritical point is mapped to the critical value is called its \emph{depth}, i.e $|w|+1$. Finally, $X_\nu$ contains all sequences over $\{0,\1\}$, except those of the form $w\nu$: this implies that the only preimage of $\nu$ is $\star\nu$. We will refer to elements of the space $X_\nu$ as \emph{itineraries}. We also call $\nu$ a precritical point; it has depth $0$.

One can think of this space as a symbolic description of the Julia set. 

We define an inverse distance between any two points in $X_\nu$ in terms of the \emph{difference function} $\diff: X_\nu \times X_\nu \to \N \cup \infty$ where $\diff(a,b)$ is defined as the position of the first difference in the sequences $a$, $b$, where an entry $\star$ counts as ``wild card symbol'' that is not different from $\0$ or $\1$.

Note that the case $\diff(a,b)=\infty$ can occur for distinct itineraries $a,b$ iff the kneading sequence $\nu$ is $\star$-periodic. This comes from the natural situation that the Hubbard tree runs through Fatou components, but for our construction it causes some technical problems; these are solved by introducing additional spaces called \emph{Fatou intervals}.

\begin{definition}[Fatou intervals]
Let $\nu$ be a $\star$-periodic kneading sequence. For a precritical point $x$ (possibly equal to $\nu$) and a sequence $w'' \in X_\nu \cap A_d^\infty$ such that $\diff(x, w'')=\infty$, we define the \emph{Fatou interval} 
\[ 
[w'', x ]
\] 
as abstract interval homeomorphic to $[0,1]$ that is disjoint from $X_\nu$, except for the endpoints. 

We define the $\sigma$ map on such an interval as a homeomorphism to the Fatou interval $[\sigma(w''), \sigma(x)]$ (the image $\sigma(x)$ is still a precritical point because $\nu$ is $\star$-periodic). For a given  kneading sequence $\nu$ we denote the space of all Fatou intervals by $F_\nu$, with the convention $F_\nu=\emptyset$ if $\nu$ is non-periodic. 
\end{definition}

In \cite{MalteDierkTrees} we construct the \emph{Hubbard tree} for any non-trivial kneading sequence. Formally, a \emph{tree} is a topological space where any two distinct points $a,b$ are connected by a unique \emph{path} $[a,b]$. A \emph{path} is a subspace homeomorphic to $[0,1]$. A point disconnecting the tree into three or more components is called a \emph{branch point}. A point that does not disconnect the tree when removed is an \emph{endpoint}. A tree is \emph{finite} if it has finitely many endpoints. Our Hubbard trees may well be infinite trees.

\begin{theorem}[The Hubbard Tree \cite{MalteDierkTrees}]
\label{Thm:Hubbard_tree}
For any non-trivial kneading sequence $\nu$, there exists a unique topological space $\Hub(\nu)\subseteq X_\nu \cup F_\nu$ with the following properties
\begin{enumerate}
\label{item:tree}
\item $\Hub(\nu)$ is a tree that satisfies $\sigma(\Hub(\nu))=\Hub(\nu)$. It contains the critical point and thus all postcritical points. 
\item 
\label{Item:HubEndpoints}
Each endpoint of the tree is a postcritical point, i.e.\ one of the points $\sigma^k(\nu)$ for $k\ge 0$. There are two possibilities:
\begin{itemize}
\item
the tree has finitely many endpoints; if $\kappa$ is this number of endpoints, then the endpoints are exactly the points $\nu$, $\sigma\nu$, $\sigma^2\nu$, \ldots , $\sigma^{\kappa-1}\nu$, and no others. In this case, the tree is compact.
\item
The tree has infinitely many endpoints, and these are exactly the points $\sigma^k\nu$ for $k\ge 0$. 
\end{itemize}
In particular, the critical value $\nu$ is always an endpoint. 

\item 
\label{Item:Hubinj}
$\sigma$ is injective on every connected subset that does not contain the critical point $\star\nu$.
\item
\label{item:alpha}
The path $[\star\nu, \nu] \subseteq \Hub(\nu)$ contains a fixed point $\alpha:=\ovl\1$.
\item
\label{item:branching}
The branch points of $\Hub(\nu)$ are periodic or preperiodic sequences in $X_\nu \cap \{\0,\1\}^\infty$. 
\item 
\label{Item:Diff}
If $a,b\in\Hub(\nu)\cap X_\nu$ and  $c \in [a,b]\cap X_\nu$, then $\diff(a,c)\ge \diff(a,b)$.
\item 
\label{Item:PrecDenseHub}
For each  $a\in \Hub(\nu)\cap X_\nu$ and every component $K$ of $\Hub(\nu)\setminus \{a\}$, either $K$ contains a Fatou interval with $a$ on its boundary, or there exist precritical points $a_n \in K$ with $\diff(a_n,a)<\infty$ and $\diff(a_n,a)\to \infty$ (this implies that the $a_n$ converge to $a$ in the topology of $\Hub(\nu)$). 
In particular, precritical points are dense in $\Hub(\nu)$ if $\nu$ is non-periodic. 
\item
\label{Item:Hubcrit}
The path $[\star\nu,\nu]\subset\Hub(\nu)$ is called the \emph{critical path}. It generates the entire tree in the sense that
\begin{equation*}
\Hub(\nu)=\bigcup_{k=0}^\infty \sigma^k([\star\nu, \nu])
\;.
\end{equation*}
\end{enumerate}
\end{theorem}

The construction of the Hubbard tree in \cite{MalteDierkTrees} is in fact done in reverse order: we first construct the critical path out of the kneading sequence, and then the Hubbard tree as its forward orbit as in \eqref{Item:Hubcrit}. 

\Newpage

\section{Parameter space}
\label{Sec:parameter}

Denote by $\K$ the space of  kneading sequences ($\star$-periodic or non-periodic). In this chapter, we provide structure to this space that allows us to compare different kneading sequences with their dynamics. We define, for $\e\in\{\0,\1\}$, the projections $\pi_{\e}: \K \to \{\0,\1\}^\infty$ as the mapping that replaces every $\star$ by $\e $. We then define the difference $\Diff: \K \times \K \to \N \cup \{+\infty\}$  by
\[
\Diff(\nu, \nu')=\max_{\e \in \{\0,\1\}} \diff(\pi_{\e}(\nu), \pi_{\e}(\nu')).
\]
The topology on $\K$ is defined by the neighborhod basis
\[
N_k(\nu)=\{ \nu' \in \K : \, \Diff(\nu,\nu')\ge k\}.
\]
Note that this is slightly different from $\diff(\cdot,\cdot)$ defined on $X_\nu$: for $\Diff$ every $\star$ must be replaced consistently by either $\0$ or $\1$, while for $\diff$ the $\star$ symbol acts as a wild card that does not differ from $\0$ or from $\1$;  for example 
 $\diff(\ovl{\1\star},\1\0\,\1\1\,\1\1\,\1\0\ldots )>8$, while $\Diff(\ovl{\1\star},\1\0\,\1\1\,\1\1\,\1\0\ldots )=4$. 

Just as for the difference between itineraries, one might ask when $\Diff(\mu, \nu)=\infty$ for two distinct kneading sequences $\mu, \nu$. This leads to the concept of \emph{bifurcation}.

\begin{definition}[Bifurcating $\star$-periodic kneading sequences]
We say that a $\star$-periodic kneading sequence $\nu=\ovl{\nu_1\nu_2\ldots  \nu_{p-1}\star}$ of period $p$ is a \emph{bifurcation} from period $q$ if $q$ strictly divides $p$ and there is a symbol $\e\in \sym$ such that $\ovl{\nu_1\nu_2\ldots  \nu_{p-1}\e}$ has exact period $q$ (but not period lower than $q$).
\end{definition}

It follows directly that if a sequence $\nu$ is a bifurcation from a sequence $\nu'$, then $\Diff(\nu,\nu')=\infty$. The same is true if $\nu$ and $\nu'$ are bifurcations from a common base sequence.

The converse is also true: if $\Diff(\nu,\nu')=\infty$ for $\nu,\nu'\in\K$, then at least one of these sequences must contain a $\star$, so it is $\star$-periodic, and then the other sequence must be periodic and thus $\star$-periodic as well. If $\nu_0=\pi_{\e}(\nu)=\pi_{\e}(\nu')$, then the period of $\nu_0$ must divide the periods of $\nu$ and $\nu'$ (possibly be equal to one of them), and so $\nu$ and $\nu'$ arise from $\nu_0$ by replacing certain symbols $\e$ by $\star$. 

We distinguish between \emph{standard} bifurcations and \emph{non-standard} bifurcations: the former are realized by complex parameters while the latter are not. More details on them are given in \cite[Proposition~4.3]{MalteDierkTrees}.

\emph{Internal addresses} are a good way to bring structure to parameter space; see  \cite{IntAddr}.

\begin{definition}[Internal address/upper and lower periodic sequence]
For every sequence $\nu \in \K\cup \sym^\infty$ starting with $\1$ we define its \emph{internal address} $I(\nu)= S_0\to S_1\to \ldots  $ as a finite or infinite, strictly increasing sequence of positive integers. They are recursively defined as follows.

\begin{itemize}
\item $S_0=1$, $\nu_0:=\ovl \1$; 
\item if $\nu$ is periodic and its exact period coincides with the exact period of $\nu_n$, then $I(\nu)$ is finite and its last entry is this period;
\item otherwise, let $S_{n+1}:=\diff(\nu,\nu_n)$. 
Let $\nu_{n+1}$ be the unique $S_{n+1}$-periodic sequence in $\{\0,\1,\}^\infty$ that first differs from $\nu_{n}$ at position $S_{n+1}$.
\end{itemize}

For a $\star$-periodic sequence $\nu$ with period $p$, there is exactly one choice for $\e\in\{\0,\1\}$ such that the internal address of $\pi_\e(\nu)$ is finite with last entry $S_n=p$. We will call this sequence $\pi_\e$ the \emph{upper sequence of $\nu$} and denote it by $\up\nu$. For the other choice $\e'\neq\e$, the sequence $\pi_{\e'}(\nu)$ is called the \emph{lower sequence} and denoted by $\lo\nu$ (in this case, the internal address is either infinite, or it is finite with last entry $S_n$ strictly dividing $p$). 
\end{definition}

For example, the kneading sequence $\nu=\1\,\ovl{\1\0}$ has infinite internal address $1-3-5-7-\ldots $. The kneading sequence $\eta=\ovl{\1\1\0\1\star}$ has internal address $1-3-5$; with $\up\eta=\ovl{\1\1\0\1\0}$. Its lower sequence $\low\eta=\ovl{\1\1\0\1\1}$ has the infinite address $1-3-6-8-11-13-16-\ldots $.

The sequence $\mu=\ovl{\1\0\1\star}$ has finite internal address $1-2-4$ and $\up\mu=\ovl{\1\0\1\1}$, while $\lo\mu=\ovl{\1\0}$ and $I(\lo\mu)=1-2$. Note that $\mu$ is a bifurcation from $\ovl{\1\star}$ with period~$2$. 

Next we define characteristic points in Hubbard trees. These are well known to have great significance for periodic orbits, but our definition applies (and is useful) also for non-periodic points.

\begin{definition}[Characteristic point]
A point $\tau\in\Hub(\nu)$ is called a \emph{characteristic point} if a single component of $\Hub(\nu)\sm\{\tau\}$ contains the critical point and all points on the orbit of $\tau$ (except $\tau$ itself when $\tau$ is periodic), but not the critical value. 
\end{definition}

Consequently, a characteristic point $\tau$ lies on the critical path of $\nu$, and moreover the path from $\nu$ to any point in the orbit of $\tau$ contains $\tau$. The critical value $\nu$ is always characteristic. 

Next, we define a  binary relation between kneading sequences; we will show in Theorem~\ref{Thm:transitive} that this is a partial order with dynamical significance.

\begin{definition}[Order of kneading sequences]
\label{Def:OrderKneadings}
For kneading sequences $\nu$  and $\mu$, we say that $\mu\prec\nu$ if $\Hub(\nu)\sm\{\nu\}$ contains a characteristic point with itinerary $\mu$ (resp.\ with itinerary $\up \mu$ in case that $\mu$ is $\star$-periodic). We denote by $\Hub_{{\tilde{\mu}}}(\nu)$ the subtree of $\Hub(\nu)$ spanned by the orbit of $\mu$ (resp.\ $\up\mu$).
\end{definition}

We use the usual notation $\mu \preceq \nu \Longleftrightarrow (\mu \prec \nu \text{ or } \mu=\nu)$ and $\nu \succ \mu \Longleftrightarrow \mu \prec \nu$. 

It is shown in \cite[Lemma~3.14]{MalteDierkTrees} that every periodic point $x\in\Hub(\nu)$ that is not an endpoint has a unique characteristic point on its periodic orbit, and all characteristic points are on $[\alpha,\nu]$. 

We say that a point $y\in \Hub(\nu)$  \emph{lies behind} $x$ if it is contained in a different component of $\Hub(\nu)\sm\{x\}$ than $\star\nu$ (that is, if $x$ separates $y$ from $\star\nu$).

\subsection{The Inclusion Theorem}

We say that a finite word $w$ over $\{\0,\1\}$ is \emph{a word in the Hubbard tree $\Hub(\nu)$} if the precritical point $w\star\nu$ is contained in $\Hub(\nu)$.

\begin{theorem}[Inclusion of trees]
\label{Thm:InclusionTrees}
Let $\nu$ and $\mu \neq \nu$ be two kneading sequences such that $\Hub(\nu)$ contains a characteristic itinerary $\tilde{\mu}\in \{ \mu, \up\mu, \lo\mu\}$ (the latter two cases  apply if $\mu$ is $\star$-periodic). Let $\Hub_{{\tilde{\mu}}}(\nu)\subset \Hub(\nu)$ denote the subtree spanned by $\tilde{\mu}$. 
Then we have the following:
\begin{enumerate}
\item
$\sigma(\Hub_{{\tilde{\mu}}}(\nu))=\Hub_{{\tilde{\mu}}}(\nu)\cup[\tilde{\mu},\nu]$. \label{item:it_subtree}
\item
\label{item:branch}
All branch points $\tau$ of $\Hub(\mu)$ appear in $\Hub(\nu)$ with the same itinerary and relative ordering. If $\tau$ is a periodic branch point that is not a preimage of $\lo\mu$, then its number of branches is the same in both trees.
\item
\label{lomu}
If $\tilde{\mu}=\up\mu$, then $\Hub(\nu)$ also contains $\lo\mu$ as characteristic point.
\item
\label{it1}
If $w\star\mu$ is a precritical point on $\Hub(\mu)$, then there is a precritical point $w\star \nu$ on $\Hub_{{\tilde{\mu}}}(\nu)$.
\item
\label{it2}
This identification respects appearance on the critical path, and the order along this path.
\item
\label{it3}
Conversely, if $w\star\nu$ is a precritical point on the critical path of $\Hub(\nu)$ with the property that it is never mapped behind $\tilde{\mu}$ before reaching the critical value, then there is a precritical point $w\star\mu\in\Hub(\mu)$.
\item
\label{item:combinatorics}
Every precritical point on the critical path of $\nu$ has the form \[w_0e_0X_{l_0}w_1e_1X_{l_1}\ldots  X_{l_{s-1}}w_s\star\nu  \text{ or } X_{l_0}w_1e_1X_{l_1}\ldots  X_{l_{s-1}}w_s
\star\nu\] where $s\ge 0$, $l_i\ge \diff(\tilde{\mu},\nu)-1$, $e_i\in \{\0,\1\}$,  the word $w_i$ is  from the Hubbard tree of $\mu$  and $X_l$ stands for the first $l$ entries of $\nu$ (or more precisely the itinerary of the local arm at $\nu$).
\item
\label{item:combinatorics_reversed}
Conversely, for $l=\diff(\tilde{\mu}, \nu)-1$ and every $s\ge 0$, there exists an injective map from  $(s+1)$-tuples of  precritical points $w_i \star \nu$ lying on $\sigma^{l}([\tilde{\mu}, \nu])\cap [\star\nu, \tilde{\mu}]$ that stay inside $\Hub_{{\tilde{\mu}}}(\nu)$ before reaching the critical value to precritical points on the critical path of $\nu$. These have the form
\[w_0e_0X_{l}w_1e_1X_{l}\ldots  X_{l}w_s\star\nu\]
for certain values $e_0,...,e_{s-1}\in \{\0,\1\}$. 
\end{enumerate}
\end{theorem}

This theorem applies in particular when $\mu \prec \nu$ (see Definition~\ref{Def:OrderKneadings}).
\begin{proof}
If $x\in\Hub_{{\tilde{\mu}}}(\nu)$, then by definition there are two points $\mu',\mu''$ on the orbit of $\tilde{\mu}$ with $x\in[\mu',\mu'']$. If $\star\nu\not\in[\mu',\mu'']$, then $\sigma$ sends $[\mu',\mu'']$ homeomorphically to its image, which is thus contained in $\Hub_{{\tilde{\mu}}}(\nu)$. Otherwise, $\sigma([\mu',\mu''])=[\sigma(\mu'),\nu]\cup[\sigma(\mu''),\nu]=[\sigma(\mu'),\tilde{\mu}]\cup[\sigma(\mu''),\tilde{\mu}]\cup[\tilde{\mu},\nu]$ because $\tilde{\mu}$ is characteristic. This proves \eqref{item:it_subtree}.

Since  $\mu$ is characteristic, we have $\tilde{\mu}\in[\alpha,\nu]$ and the preimage $\1\tilde{\mu}\in[\star\nu, \alpha]$. 

Now we prove \eqref{it1} for all precritical points on the critical path of $\Hub(\mu)$. In \cite[Definition 2.4]{MalteDierkTrees}, these precritical points are constructed recursively by constructing the unique precritical point on $[\star\mu,\mu]$ that first hits $\star\mu$, then proceeding recursively for the two sub-intervals on the two sides of the precritical point just constructed. We show that precritical points with the same prefix words can be constructed on the critical path of $\Hub(\nu)$.

The base of the induction is provided by the first precritical point on $[\star\mu,\mu]$, say $w_0\star\mu$. It is found by iterating $[\star\mu,\mu]$ as long as this iteration is injective, that is until the sequences $\star\mu$ and $\mu$ have a different initial symbol after the iterated shift. Note that after the first iteration, the image of $[\star\mu,\mu]$ is $[\mu,\sigma\mu]$.

Set $k:=\diff(\star\mu,\mu)$. Then the length of $w_0$ equals $\diff(\star\mu,\mu)-1$, and the interval $[\star\mu,\mu]$ maps forward $k-1$ iterations, after which the critical point $\star\nu$ cuts the image interval into the two sub-intervals $[\sigma^{k-2}\mu,\star\mu]$ and $[\star\mu, \sigma^{k-1}\mu]$. 
Their immediate images after one further (injective) iteration are $[\sigma^{k-1}\mu,\mu]$ and $[\sigma^{k}\mu,\mu]$.

By analogy in $\Hub(\nu)$ we iterate the interval $[\1\tilde{\mu},\tilde{\mu}]$ as long as this iteration is injective, that is until the sequences $\1\tilde{\mu}$ and $\tilde{\mu}$ have a different initial symbol after the iterated shift. Again, after one iteration we have the interval $[\tilde{\mu},\sigma\tilde{\mu}]$, and we iterate until the first difference is found. This happens after the same number of iterations as in $\Hub(\mu)$: this is clear if $\mu$ is non-periodic, and otherwise this follows from the fact that $[\lo\mu, \mu]\subset \Hub(\mu)$ does not contain precritical points. The precritical point found in $\Hub(\nu)$ has itinerary $w_0\star\nu$, and the two sub-intervals are $[\sigma^{k-2}\tilde{\mu},\star\nu]$ and $[\star\nu, \sigma^{k-1}\tilde{\mu}]$. Their immediate images are $[\sigma^{k-1}\tilde{\mu},\nu]=[\sigma^{k-1}\tilde{\mu},\tilde{\mu}]\cup[\tilde{\mu},\nu]$ and  $[\sigma^{k}\tilde{\mu},\nu^*]=[\sigma^{k}\tilde{\mu},\tilde{\mu}]\cup[\tilde{\mu},\nu]$. 

The precritical point $w_0\star\nu$ does not map during iteration to the part of $\Hub(\nu)$ behind $\mu$ before it lands on the critical value $\nu$ because it is ``sandwiched'' between two points on the orbit of $\mu$: initially we have $w_0\star\nu\in[\1\tilde{\mu},\tilde{\mu}]$, and this sandwiching relation is preserved as long as the interval $[\1\tilde{\mu},\tilde{\mu}]$ maps forward injectively, that is until it maps over $\star\nu$ and the point $w_0\star\nu$ is found.

For the inductive step, consider precritical points $w_1\star\mu$ and $w_2\star\mu$ on the critical path of $\mu$ that construct the child itinerary $w_3\star \mu$ between them and let $|w_2|$ be the higher depth (in an edge case, the point $w_1\star\mu$ may be $\mu$).
By inductive hypothesis, we assume that the precritical points $w_1\star\mu$ and $w_2\star\mu$ have analogues $w_1\star\nu$ and $w_2\star\nu$ in $[\star\nu, \tilde{\mu}]\subset \Hub(\nu)$ such that $w_1\star\nu$ and $w_2\star\nu$ do not map behind $\tilde{\mu}$ before reaching $\nu$. In $\Hub(\mu)$ we iterate the interval $[w_1\star\mu, w_2\star\mu]$ injectively $|w_2|+1$ times until we have $[\sigma^t\mu,\mu]$ (where $t$ is the difference between the depths). The argument is the same as for the base case, where we had $t=1$. In order to find $w_3\star\mu$, we iterate $[\sigma^t\mu,\mu]$ as long as it is injective. Setting $k:=\diff(\star\mu,\sigma^{t-1}\mu)$, the number of injective iterations of $[\sigma^{t}\mu,\mu]$ is $k-2$. After the last injective iteration, we land at the interval $[\sigma^{t+k-2}\mu,\sigma^{k-2}\mu]\ni\star\mu$. 

The new sub-intervals are $[\sigma^{k-2}\mu,\star\mu]$ and $[\star\mu, \sigma^{t+k-2}\mu]$, and their immediate images after one further iteration are $[\sigma^{k-1}\mu,\mu]$ and $[\sigma^{t+k-1}\mu,\mu]$. 

By analogy in $\Hub(\nu)$, we iterate $[w_1\star\nu, w_2\star\nu]$ as long as it is injective. First, after $|w_1|+1$ iterations, we land at $[\nu, \sigma^{|w_1|+1} w_2 \star \nu]$. Since by assumption, $\sigma^{|w_1|+1} w_2 \star \nu$ is a precritical point not behind $\tilde{\mu}$, the interval $[\nu, \sigma^{|w_1|+1} w_2 \star \nu]$ contains $[\tilde{\mu}, \sigma^{|w_1|+1} w_2 \star \nu]$. Iterating only the latter interval further, we obtain $[\sigma^{t}\tilde{\mu}, \nu]$, where $t$ is the same as for $\Hub(\mu)$. Since $\tilde{\mu}$ is characteristic, this interval contains $[\sigma^{t}\tilde{\mu}, \tilde{\mu}]$.

Again this interval can be iterated the same number of times in $\Hub(\nu)$ as in $\Hub(\mu)$ and we once again obtain the two subintervals $[\sigma^{k-2}\tilde{\mu},\star\tilde{\mu}]$ and $[\star\tilde{\mu}, \sigma^{t+k-2}\tilde{\mu}]$. Thus, the precritical point $w_3 \star \nu$ has to lie between $w_1\star\nu$ and $w_2\star\nu$. Since $w_3\star\nu$ is sandwiched first by the two precritical points and then by orbit points of $\mu$, it does not land behind $\tilde{\mu}$ either before reaching the critical point. This concludes the proof of \eqref{it1} and \eqref{it2} along the critical path.

To show \eqref{it1} and \eqref{it2} for a precritical point anywhere on $\Hub(\mu)$, simply note that any precritical point on the Hubbard tree is an image of a precritical point on the critical path (see Theorem~\ref{Thm:Hubbard_tree}~\eqref{Item:Hubcrit}); this implies that the order-preserving identification of precritical points on the critical path extends to the whole tree. 

To show $\eqref{item:branch}$ we distinguish two cases. If a branch point  $\tau \in \Hub(\mu)$ is not an endpoint of a Fatou interval, each branch contains a sequence of precritical points that converge to $\tau$. Hence, by $\eqref{it1}$, the itinerary $\tau$ also exists in $\Hub(\nu)$.  If the branch point is periodic, then each local arm is an image of the two local arms of $\tau$ lying on the critical path by \cite[Corollary 3.15]{MalteDierkTrees}. Since precritical points of $\mu$ lie dense on them, their dynamics and in particular the number of branches at $\tau$ has to stay the same.

The other case is that $\tau$ is an endpoint of a Fatou interval. Then $\mu$ must be a bifurcation and $\tau$ is an iterated preimage of $\lo\mu$. Let $q$ be the period of $\lo\mu$. Then the subtree of $\Hub(\nu)$ spanned by the points $\{\sigma^{kp}\tilde{\mu}\}_{k\ge 0}$ must contain precritical points that coincide with $\lo\mu$ except at positions that are multiples of $q$ and these points thus have to converge to $\tau$. A similar reasoning applies for preimages of $\lo\mu$. These arguments also yield \eqref{lomu}.

To show $\eqref{it3}$, note that the critical path of $\Hub(\nu)$ equals $[\tilde{\mu},\nu]\cup[1\tilde{\mu},\tilde{\mu}]\cup[\star\nu, 1\tilde{\mu}]$. Therefore, every precritical point $w\star\nu$ on the critical path of $\nu$ either lies on $[\tilde{\mu},\nu]\cup [\star\nu, 1\tilde{\mu}]$, or it lies on $[w_1\star\nu, w_2\star\nu]$ for some pair of precritical points as discussed in the induction step above. The only way how they can be omitted in the inductive step above is if they are mapped to $[\tilde{\mu}, \nu]$ once we iterated to $[\nu, \sigma^{|w_1|+1} w_2 \star \nu]$ or once we iterated to $[\sigma^{t}\tilde{\mu}, \nu]$. 

So if $w\star\nu$ is such an omitted precritical point, then $w$ has to coincide with $w_1$ or $w_2$ one step before they are mapped to the new sub-intervals. When $w_1\star\nu$ or $w_2\star\nu$, respectively, would hit the critical point, there is a choice on which side of it $w\star\nu$ lands, and this choice is displayed by $e_i$. After landing in $[\tilde{\mu},\nu]$, the next $\diff(\tilde{\mu},\nu)-1$ entries are determined because this is how long $[\tilde{\mu},\nu]$ is mapped injectively. After that, the iterates of $w\star\nu$ stay outside of $H_{\tilde{\mu}}(\nu)$ and lie on the same side of the critical point as the corresponding iterate of $\nu$, or they land inside of $H_{\tilde{\mu}}(\nu)$ and the argument can be repeated. This shows \eqref{item:combinatorics}.

We also see that for each precritical point $w_0\star\mu$ on the critical path of $\mu$, there exists an interval of itineraries starting with $w_0$ on the critical path of $\nu$ that is homeomorphically mapped onto $[\nu, \tilde{\mu}]$ under $\sigma^{|w_0|+1}$. After $l$ further iterations, this interval is mapped to $[\sigma^l \nu, \sigma^l \tilde{\mu}]\ni \star\nu$. This interval then contains precritical points on the critical path that also exist in $\Hub(\nu)$ and the argument can be repeated for $w_1$. This yields that every representation in \eqref{item:combinatorics_reversed} is assumed by a precritical point of $\nu$, the map is injective because the positions of the words $X_l$ indicate each time the point enters $[\tilde{\mu}, \nu]$. 
\end{proof}

\begin{theorem}[Order among kneading sequences]
\label{Thm:transitive}
The relation $\prec$ is a strict partial order among kneading sequences.
\end{theorem}
\begin{proof}
Clearly, $\nu \prec \nu$ does not hold because itineraries are unique in a Hubbard tree and $\up\nu$ does not occur on the critical path of $\star$-periodic kneading sequences by construction. Thus, it suffices to prove transitivity of the relation. Let $\nu_1,\nu_2,\nu_3$ be three arbitrary kneading sequences satisfying $\nu_1\prec \nu_2$ and $\nu_2 \prec \nu_3$. We wish to conclude that $\nu_1 \prec \nu_3$. 

First suppose that $\nu_2$ is not a standard bifurcation of $\nu_1$. Then $\up{\nu_1}$ is a limit point of precritical points on the critical path of $\nu_2$. By Theorem~\ref{Thm:InclusionTrees}, every precritical point of $\nu_2$ can be found on the critical path of $\nu_3$, hence, $\up{\nu_1}$ lies on the critical path of $\nu_3$. 

Now let $\nu_2$ be a standard bifurcation of period $p$ with base sequence $\nu_1$ which has period $q$. If $p/q>2$, then $\nu_1$ is a branch point and thus also appears in $\Hub(\nu_3)$ by Theorem~\ref{Thm:InclusionTrees}~\eqref{item:branch}. If $p=2q$, then $[\up\nu_2, \sigma^q \up\nu_2]$ is mapped onto itself under $\sigma^q$ with reverted orientation. Thus, $\nu_1$ appears as fixed point. 

Since none of the precritical points also found in $\Hub(\nu_2)$ map behind $\nu_2$ (resp. $\up{\nu_2}$) in $\Hub(\nu_3)$, the sequence $\up{\nu_1}$ does not map behind $\up{\nu_2}$ either. Thus since $\up{\nu_1}$ is characteristic in $\Hub(\nu_2)$, it therefore has to be characteristic in $\Hub(\nu_3)$, too. Hence, $\nu_1 \prec \nu_3$.
\end{proof}

\subsection{Extended Hubbard trees}

We say two kneading distinct sequences $\nu$ and $\mu$ are \emph{comparable} if $\nu\prec \mu$ or $\mu \prec \nu$. In that case, the inclusion theorem tells us that the Hubbard tree of the smaller sequence is in some sense included inside the tree of the larger sequence, so one cne could interpret the partial order ``$\prec$'' on kneading sequences as subset relation ``$\subset$'' on Hubbard trees.

However, if the sequences are not comparable, we can extend each tree so as to contain a point whose itinerary is the other kneading sequence. 
To do this, let $\nu$  be a kneading sequence and $\mu \in X_\nu \cap \sym^\infty \setminus \{\up\nu \}$ be an itinerary. Just as for the critical path $[\star\nu, \nu]$  (see  \cite[Section 2.3]{MalteDierkTrees}), we can construct the interval $[\mu, \nu]$ recursively: first by determining the first precritical point $w\star\nu$ between the itineraries $\mu$ and $\nu$, and by repeating the step between neighboring itineraries if their symbolic difference is finite, otherwise by adding a Fatou interval. After taking the closure, one obtains a linearly ordered set $[\mu, \nu]$ of itineraries (and possibly Fatou intervals) homeomorphic to the unit interval $[0,1]$. 

\begin{definition}[Extended Hubbard tree]
Let $\nu$  be a kneading sequence and $\mu \in X_\nu \cap \sym^\infty \setminus \{\up\nu\}$. We define the \emph{Hubbard tree of $\nu$ extended to $\mu$} as the following set in $X_\nu \cup F_\nu$:
\[
\Hub(\nu) [\mu] := \Hub(\nu) \cup \bigcup_{n\ge 0}^\infty \sigma^n( [\mu, \nu])
\;.
\]
\end{definition}

\begin{proposition}[Extended Hubbard tree]
The set $ \Hub(\nu) [\mu]$ is $\sigma$-invariant. If $\mu \in \Hub(\nu)$, then $\Hub(\nu) [\mu]=\Hub(\nu)$. If $\mu$ is (pre)periodic, then $\Hub(\nu) [\mu]$ is a tree. 
\end{proposition}
\begin{proof}
The set is invariant by construction. Since the construction of the interval $[\mu, \nu]$ is recursive, if two points of $[\mu, \nu]$ appear in $\Hub(\nu)$, so do all points between it. So if $\mu$ lies already in the tree, nothing new is created. Each of the sets  $\sigma^n( [\nu, \mu])$ contain at most one interval not already in $\Hub(\nu)$, which goes from a postcritical point to an orbit point of $\mu$. And each interval added adds at most one branch to the tree. Since $\mu$ has a finite orbit, $\Hub(\nu)  [\mu]$ is a tree. 
\end{proof}

The statement should also hold for $\mu$ that are neither periodic nor preperiodic, but the result stated suffices for our purposes.

\subsection{The Branch Theorem}

We will need a technical lemma from \cite{HenkDierkPreprint} which makes use of the symbolic notion of $\rho$-functions: For every infinite sequence $\nu=\1\nu_2.... \in \{\0,\1\}^\infty$ define $\rho_\nu(n):=\inf\{k>n: \, \nu_k \neq \nu_{k-n}\}$ and $\orb_{\rho}(k):=\{k, \rho(k),\rho^2(k),...\}$. 

\begin{lemma}[Combinatorics of $\rho$-orbits]
Let $\nu=\1\nu_2.... \in \{\0,\1\}^\infty$ and set $\rho=\rho_\nu$. If $m \in \orb_{\rho}(1)$ and $s<m<\rho(s)$, then $m \in \orb_\rho(\rho(m-s)-(m-s))$. 
\end{lemma}
\begin{proof} See \cite[Lemma~4.3(1)]{HenkAlexExist}.  \end{proof}

As explained in \cite{HenkAlexExist}, one can interpret the $\rho$-function in the following way: The depth of the first precritical point between $\nu_1...\nu_{s-1}\star\nu$ and $\nu$ equals $\rho(s)$. Consequently, $\rho(m-s)-(m-s)$ is the depth of the first precritical point between $\nu^{m-s}$ and $\nu$. And $\orb_\rho(1)$ simply is the internal address of $\nu$. We can thus restate the lemma as:

\begin{lemma}[Combinatorics of $\rho$-orbits, restated]
\label{Lem:rhocombs}
Let $m$ be an element of the internal address of a non-periodic kneading sequence $\nu$. Moreover, let $\zeta \in \Hub(\nu)$ be a precritical point of depth $s<m$ such that $[\zeta, \nu]$ is mapped injectively under $\sigma^m$. Then $\rho_m \in [\nu^{m-s}, \nu]$, where $\rho_m$ is the precritical point of depth $m$ that coincides with $\nu$ for more than $m$ entries. 
For $\star$-periodic kneading sequences, the statement still holds if $m$ is an element of the internal address of $\lo\nu$. 
\end{lemma}

\begin{lemma}[Monotonicity of internal address]
\label{Lem:int_add}
Let $\nu$ be a kneading sequence whose internal address starts with $S_1-S_2-\ldots  -S_{n}-S_{n+1}$ for some $n \ge 1$. Let $\mu$ be the $\star$-periodic sequence with internal address $S_1-S_2-\ldots  -S_n$. Then $\mu \prec \nu$.

Moreover, if $\nu$ is $\star$-periodic and the internal address of its lower sequence $\lo\nu$ starts with $S_1'-S_2'-\ldots  -S_{m}'$, then $\mu' \prec \nu$ where $\mu'$ is the $\star$-periodic sequence with internal address $S_1'-S_2'-\ldots  -S_m'$.

\end{lemma}
\begin{proof}
Let us first assume that $\nu$ is non-periodic. By the construction of the internal address, the critical path of $\nu$ contains precritical points $\rho_1=\star\nu \prec \rho_2 \prec \ldots   \prec \rho_n \prec \rho_{n+1} \prec \nu$ such that $\rho_j$ has depth $S_j$ and $\rho_{j+1}$ is the precritical point of lowest depth between $\rho_j$ and $\nu$. 

Set $p=S_n$. Let $\mu$ be the $p$-periodic sequence that coincides with $\nu$ for at least $p$ entries. Assume that $\mu$ does not lie on the critical path. Extend the Hubbard tree of $\nu$ to $\mu$. By definition of the internal address, we have $\rho_{n+1}\in [\mu, \nu]$. Let $\tau \in [\rho_n, \rho_{n+1}]$ be the branching point where $\mu$ branches off. The interval $[\mu, \rho_n]$ is injectively mapped by $\sigma^p$ to $[\mu, \nu]$, so the point $\sigma^p \tau$ lies on this interval. It does not lie on $[\tau, \mu]$ though, because this would violate expansivity of the interval $[\tau, \sigma^p \tau]$ which must contain a precritical point. Hence, $\sigma^p \tau \in [\tau, \nu]$. 

The interval $[\rho_n, \nu]$ can be mapped injectively under $\sigma^p$. If $\rho_n$ lies between $\nu$ and $\sigma^p \nu$, then this mapping is expanding and there must be a fixed point on $[\rho_n, \nu]$. Since this point coincides with $\nu$ for at least $S_n$ entries, it has to be $\mu$. 

Otherwise, $\sigma^p\nu$ branches off at $\sigma^p \tau$ and $\sigma^{2p}\tau \in [\sigma^p \tau, \sigma^p \nu]$. Let $Y$ denote this branch at $\sigma^p\nu$. While the interval $[\sigma^p \nu, \sigma^{2p} \tau]$ is injective under $\sigma^p$, all further $\sigma^p$-iterates of $\tau$ lie in $Y$. Since there must lie a precritical point on $[\sigma^p \nu, \sigma^{2p} \tau]$, there must exist a precritical point $\zeta \in Y$ of depth $s<p$. Moreover, we can choose $\zeta$ as close as possible to $\nu$ such that $[\zeta, \nu]$ is injective under $\sigma^p$. Hence, the subtree spanned by the tree points $\zeta, \nu$ and $\tau$ will be homeomorphically mapped by $\sigma^p$ to the subtree spanned by the points $\sigma^{p-s} \nu, \sigma^p \nu, \sigma^p\tau$. We conclude that $\sigma^{p-s}\nu \in Y$ and by Lemma~\ref{Lem:rhocombs} this implies $\rho_n \in [\sigma^{p-s}, \nu]$, but this is impossible because $\rho_n$ already lies on $[\star\nu,\tau]$. 

Hence, we can conclude that $\mu$ lies on the critical path. By \cite[Lemma~3.14]{MalteDierkTrees}, an iterate of $\mu$ is characteristic. Since the orbit point of $\mu$ closest to $\nu$ is $\mu$ itself, we conclude $\mu \prec \nu$. 

For the periodic case we have precritical points $\rho_1'=\star\nu \prec \rho_2' \prec \ldots   \prec \rho_n'  \prec \lo\nu \prec \nu$ such that $\rho_j'$ has depth $S_j'$ and $\rho_{j+1}'$ is the precritical point of lowest depth between $\rho_j$ and $\lo\nu$. The rest follows as before, as everything behind $\lo\nu$ consists of Fatou intervals. This shows the second part of the lemma. Finally note that the internal addresses of $\lo\nu$ and $\nu$ coincide as long as the entries are smaller than the period, so the first part of the lemma also holds for $\star$-periodic $\nu$. 
\end{proof}

\begin{theorem}[The Weak Branch Theorem]
\label{Thm:branch_weak}
Let $\nu$ and $\nu'$ be two kneading sequences such that $k:=\Diff(\nu,\nu')<\infty$. Then there exists a $\star$-periodic kneading sequence $\mu$ such that 
\begin{enumerate}
\item \label{it:branch1} $\mu \preceq \nu$ and $\mu \preceq \nu'$
\item \label{it:branch2} $\Diff(\mu, \nu)\ge k$ and $\Diff(\mu, \nu')\ge k$. 
\end{enumerate}
\end{theorem}
\begin{remark}
We call the theorem the weak branch theorem because the sequence $\mu$ found may not be maximal. One can show that for each pair of non-comparable $\nu$ and $\nu'$ there exists a $\star$-periodic or preperiodic kneading sequence $\mu$ that satisfies \eqref{it:branch1} and \eqref{it:branch2} but no sequence $\mu' \succ \mu$ satisfies \eqref{it:branch1}. This has been shown for $\star$-periodic kneading sequences in \cite[Theorem~8.16]{HenkDierkPreprint}. 
\end{remark}
\begin{proof}
Let $\chi$ and $\chi'$ be the lower or upper sequences of $\nu$ and $\nu'$, respectively, such that $k=\Diff(\nu,\nu')=\diff(\chi,\chi')$ (with the convention that $\chi=\nu$ if $\nu$ is non-periodic and same for $\nu'$). Since $k< \infty$, the internal addresses of $\chi$ and $\chi'$ must differ at some point. Without loss of generality, we may choose $n$ such that the first $n$ entries  $S_1-\ldots  -S_n$ of the internal addresses of $\chi$ and $\chi'$ are the same but $S_n$ is the last entry in the internal address of $\chi$ or $S_{n+1}' < S_{n+1}$, where $S_{n+1}$ and $S_{n+1}'$ are the $(n+1)$-th entries of the internal address of $\nu$ and $\nu'$, respectively. By definition of the internal address, we then have $k=\diff(\chi,\chi')=S_{n+1}'$. 

Now let $\up\mu$ be the periodic sequence with internal address $S_1-\ldots  -S_n$ and set $\mu$ to be the corresponding kneading sequence. By Lemma~\ref{Lem:int_add}, the sequence $\mu$ satisfies \eqref{it:branch1}. Moreover, $\mu=\nu$ or $\Diff(\mu, \nu)\ge \diff(\up\mu,\chi)=S_{n+1}>k$. And we also have $\Diff(\mu, \nu')\ge\diff(\up\mu, \chi')=S_{n+1}'=k$. This shows \eqref{it:branch2}.
\end{proof}

For the sake of readability, we will no longer point out the difference between a $\star$-periodic kneading sequence $\mu$ in parameter space and its representation as upper sequence $\up\mu$ in the dynamics of a Hubbard tree for the rest of the paper when there is no confusion. 

\Newpage

\section{Core entropy}
\label{Sec:entropy}

\begin{definition}[Core entropy]
\label{Def:KneadingEntropy}
For a kneading sequence $\nu$, let $N_\nu(n)$ denote the number of precritical points of depth $n$ on the critical path $[\star\nu,\nu]$, and let $\log^+(x):=\max(\log x,0)$ for $x\ge 0$. We define \emph{core entropy} of $\nu$ as follows:
\[
h(\nu):=\limsup_{n\to\infty} \frac 1 n \log^+(N_\nu(n))
\;.
\]
\end{definition}

We understand $\log$ as the natural logarithm; other choices of logarithm would change the value by a fixed constant.

If $\nu$ is $\star$-periodic or preperiodic so that $\Hub(\nu)$ is a finite tree, then this definition coincides with all other definitions of core entropy such as via the eigenvalue of the (finite) transition matrix \cite[Lemma~2.3]{DimaEntropy}.

\begin{lemma}[Basic properties of core entropy]
Core entropy of a kneading sequence $\nu$ has the following properties:
\begin{enumerate}
    \item  $0\le h(\nu) \le \log 2$.
\item $N_\nu(n)\le 2^{n-2}$ for $n\ge 2$.
    \item If $\mu \prec \nu$, then $h(\mu)\le h(\nu)$ and even $N_\mu(n)\le N_\nu(n)$ for all $n$.
\end{enumerate}
\label{Lem:entropybasics}
\end{lemma}
\begin{proof}
Recall that every interval in $\Hub(\nu)$ that does not contain the critical point is mapped injectively. If $n_1$ is the lowest depth of precritical points on the interior of the critical path, then $n_1\ge 2$ and the precritical point of depth $n_1$ must be unique. It divides the critical path into two subintervals. Each will be iterated injectively at least $n_1+1$ times, so in each of the subintervals there exists a unique precritical point of next higher depth (which may be different on the two subintervals). 

In each subdivision step, the number of subintervals doubles, while the depth increases by one or more. Therefore, $N(n)\le 2^{n-2}$, and this implies the upper bound on $h(\nu)$. 

The final claim follows directly from Theorem~\ref{Thm:InclusionTrees} \eqref{it1}: every precritical point $w\star\mu$ on the critical path of $\Hub(\mu)$ has an analogue $w\star\nu$ on the critical path of $\Hub(\nu)$ of equal depth, so for each depth the critical path in $\Hub(\nu)$ has at least as many precritical points as the critical path in $\Hub(\mu)$. In other words, $N_\mu(n)\le N_\nu(n)$ for all $n$, and this implies the claim on entropy.
\end{proof}

\begin{remark}
The $\log 2$ bound also simply follows from the fact that there are at most $2^{n-1}$ precritical points of depth $n$ that can be formed with the alphabet $\{\0,\1\}$. However, the proof presented here has the advantage that it generalizes to higher degrees; a systematic discussion can be found in \cite{MalteDierkTranscEntropy}. 
\end{remark}

The upper bound is sharp and assumed by a unique kneading sequence:
\begin{proposition}[Entropy $\log 2$]
\label{Prop:uniquemax}
$\1\overline{\0}$ is the unique kneading sequence with core entropy $\log 2$. More precisely, if $\nu\neq1\ovl\0$ and $s$ is the position of the second $\1$ in a kneading sequence $\nu$, then $h(\nu)\le \log 2 -2^{-s}/s $.
\end{proposition}
\begin{proof}
One can check that for $\nu=\1\overline{\0}$ all possible $2^{n-2}$ precritical points of depth $n$ will be constructed on the critical path, hence its core entropy is $\log 2$. (It is well known that this kneading sequence is realized by the real quadratic polynomial $x\mapsto x^2-2$ on $[-2,2]$ or equivalently by $x\mapsto 2x^2-1$ on $[0,1]$.)

However if $\nu \neq \1\overline{\0}$,  let $s\ge 2$ be the position of its second entry equal to $\1$ and define the word $w=\1\0\ldots  \0$ of length $s$. 
Then all $2^{s-1}+1$ precritical points of depths at most $s$  occur on the critical path of $\nu$, but the precritical point $w\star\nu$ of depth $s+1$ does not. Moreover, we claim that the sequence $w$ never occurs anywhere in the itinerary of any precritical point on the critical path before the $\star$ symbol. This is because precritical points are created by comparing $\nu$ with an iterate of $\nu$, but the former always starts with $\1$ but does not start with the word $w$. 

This claim implies that if we divide a word of length $n$ of a precritical point on the critical path into $\lceil n/ s \rceil$ subwords of length at most $s$, then there are at most $2^{s}-1$ possibilities to choose each subword. Thus,
\[
N(n+1)\le (2^s -1)^{n/s +1}
\]
and
\[
h(\nu) \le \lim_{n \to \infty}  \frac 1 n \log\left( (2^s -1)^{n/s +1}\right)=\log(2^s-1)/s   < \log 2 -2^{-s}/s 
\;.
\]
\end{proof}

\subsection{Core entropy on the Hubbard tree}

We have defined core entropy only using precritical points on the critical path. A different possibility would be to define $N_H(n)$ as the number of precritical points on the entire Hubbard tree $\Hub(\nu)$. This leads to the alternate definition
\[
h_{H}(\nu):=\limsup_{n\to\infty} \frac 1 n \log^+(N_H(n)).
\]
We obviously have $0\le h(\nu)\le h_H(\nu) \le \log 2$. If the Hubbard tree is finite, then $h=h_H$, see Lemma~\ref{Lem:FiniteTrees}. However, if the tree is infinite, it can happen that $h<h_H$. As explained in \cite[Example 7]{MalteDierkTrees}, a kneading sequence where every finite word $w$ over $\{\0,\1\}$ occurs somewhere in the sequence has the property that $w\star\nu$ is a precritical point on the Hubbard tree, so $h_H(\nu)=\log 2$. On the other hand, $h(\nu)<\log 2$ by Proposition~\ref{Prop:uniquemax}. There are infinitely many kneading sequences with $h_H(\nu)=\log 2$, while $h(\nu)$ may be arbitrarily small positive. This and other possible definitions of core entropy are discussed in \cite{MalteDierkTranscEntropy}.

\Newpage

\section{Estimates on precritical points}
\label{Sec:estimates}

\begin{definition}[Special classes of kneading sequences]
\label{Def:SpecialKneadings}
We define the following classes of kneading sequences.
\begin{itemize}
\item
A sequence $\nu$ is called \emph{recurrent} if it is not periodic but $\diff(\nu,\sigma^n\nu)$ is unbounded.  Otherwise, $\nu$ is called $\emph{non-recurrent}$.
\item 
Let $\kappa(\nu)\in\N\cup\{+\infty\}$ be the number of endpoints of the Hubbard tree $\Hub(\nu)$. We say that $\nu$ is \emph{tree-finite} if $\kappa(\nu)<\infty$, and \emph{tree-infinite} otherwise.
\item
we say that $\nu$ is \emph{uniformly expanding with parameter $\lambda\in\N$} if, whenever $\rho$ and $\rho'$ are two precritical points on the critical path with equal depth $n$, then the image of $[\rho,\rho']$ after $n+\lambda$ iterations  contains the critical path. 
\end{itemize}
\end{definition}

In our definition, we consider periodic sequences as non-recurrent, while preperiodic sequences are clearly non-recurrent.

\begin{lemma}[Finite trees]
\label{Lem:FiniteTrees}
If $\nu$ is tree-finite, then $\displaystyle\Hub(\nu)=\bigcup_{k=0}^{\kappa(\nu)-1} \sigma^k[\star\nu,\nu]$ and $h_H(\nu)=h(\nu)$. 
\end{lemma}
\begin{proof}
The finite union contains the paths $[\nu, \sigma\nu], [\sigma(\nu), \sigma^2(\nu)],\ldots  ,\sigma^{\kappa-1}(\nu), \sigma^\kappa(\nu)]$ connecting all endpoints with each other. Hence, they cover the entire Hubbard tree. Moreover, we conclude that every precritical point on the Hubbard tree of depth $n$ is  the image of a precritical point on the critical path of depth at most $n+\kappa-1$. Hence,
\begin{equation}
N_H(n)\le \sum_{i=0}^{\kappa-1} N(n+i).
\label{Eq:NumberEnyaTree}
\end{equation}
By definition of core entropy, we have for all $\eps>0$ and sufficiently large $n$ the bound $N(n)\le e^{(h(\nu)+\eps)n}$. The inequality above then yields $h_H(\nu)\le h(\nu)+\eps$ and hence $h_H(\nu)=h(\nu)$. 
\end{proof}

\begin{lemma}[Monotonicity of $\kappa$]
\label{Lem:KappaMonotone}
If $\mu \prec \nu$, then $\kappa(\mu)\le\kappa(\nu)$.
\end{lemma}
\begin{proof}
Since all branch points of $\Hub(\mu)$ are also branch points with at least as many branches in $\Hub(\nu)$ by Theorem~\ref{Thm:InclusionTrees}~\eqref{item:branch}, the Hubbard tree of $\mu$ cannot have more endpoints than $\Hub(\nu)$.
\end{proof}

\begin{remark}
It turns out that if $\kappa(\nu)=\infty$, then $\nu$ is maximal in the sense that there is no $\nu'\succ \nu$: whenever $\nu'\succ\nu$ and both are non-periodic, then it is easy to see that there exists a $\star$-periodic $\mu$ with $\nu'\succ\mu\succ\nu$; then $\kappa(\nu)\le\kappa(\mu)<\infty$. 

\end{remark}

In the following lemma, and elsewhere when there is no danger of confusion, we write for simplicity $h=h(\nu)$, $\kappa=\kappa(\nu)$, and $\lambda=\lambda(\nu)$. 

\begin{lemma}[Bound for uniformly expanding sequences]
\label{Lem:UpperBoundEnya}
If $\nu$ is uniformly expanding with parameter $\lambda$, then $N(n)\le  e^{h (n+\lambda)}$ for all $n$. Moreover, $N_H(n) \le \kappa e^{h(n+\lambda+\kappa)}$.
\end{lemma}
The last claim is of course meaningless if $\kappa=\infty$.

In the proof, we need the following concept: a closed interval $I\subset\Hub(\nu)$ is an \emph{$n$-horseshoe} for $\sigma^k$ if $I$ contains $n$ closed sub-intervals such that $\sigma^k$ maps each of them onto a set that contains $I$. It is well known that this implies that the topological entropy of $\sigma$, hence the core entropy of $\nu$, is then at least $(\log n)/k$ \cite[Theorem~A]{horseshoes}.

\begin{proof}
By definition, the interior of the critical path contains $N(n)$ precritical points of depth exactly $n$. These $N(n)$ precritical points divide the critical path into exactly $N(n)+1$ different intervals. By definition of uniform expansivity, each of these is sent by $\sigma^{n+\lambda}$ to an image that contains the critical path. This means that the critical path itself is an $N(n)+1$-horseshoe for $\sigma^{n+\lambda}$, which implies
\[
h=h(\nu) \ge \frac{\log(N(n)+1)}{n+\lambda} \qquad \text{and thus} \qquad
 N(n) < e^{h(n+\lambda)}
\;.
\]
Using \eqref{Eq:NumberEnyaTree} in Lemma~\ref{Lem:FiniteTrees}, we then have
\[
N_H(n)\le \sum_{i=0}^{\kappa-1} N(n+i) \le \sum_{i=0}^{\kappa-1} e^{h(n+\lambda+i)} \le  \kappa e^{h(n+\lambda+\kappa)}
\;.
\]
\end{proof}

\goodbreak

We need to introduce one more relevant kind of kneading sequences.

\begin{definition}[Renormalizable kneading sequence]
\label{Def:Renormalizable}
A kneading sequence $\nu$ is called \emph{$p$-renormalizable} (with period $p\ge 2$) if the subtree $K \subset \Hub(\nu)$ spanned by the points $\sigma^{kp}(\nu)$ for $k\ge 0$ has the following properties:
\begin{enumerate}
\item
\label{Item:RenormStrictSubtree}
it contains more than one point, but not all of $\Hub(\nu)$;
\item
\label{Item:RenormInvariant}
it is invariant under $\sigma^p$ in the sense that $\sigma^p(K)\subset K$;
\item 
\label{Item:RenormPeriodic}
finally, if $\nu$ is $\star$-periodic, then its period is greater than $p$. 
\end{enumerate}
In this case, we call $K$ the \emph{little Hubbard tree of $\nu$} (with respect to $p$-renormalization).
\end{definition}

By definition, the subtree $K$ is uniquely determined in terms of $\nu$ and $p$ whenever it exists (because its endpoints are). We will see later that $\star\nu\in \sigma^{p-1}(K)$, so we can change the invariance in \eqref{Item:RenormInvariant} to $\sigma^p(K)=K$.

Note that $K$ cannot contain $\star\nu$ and thus not the entire critical path: if it did, it would contain all sequences $\sigma^{kp}(\star\nu)$ for $k\ge 0$ by $\sigma^p$-invariance, so $\sigma(K)$ would contain all sequences $\sigma^{kp}(\nu)$, thus $\sigma(K)\supseteq K$; this implies by induction $\sigma^p(K)\supseteq \sigma(K)\supseteq K\supseteq \sigma^p(K)$, hence $\sigma(K)=K$. But if $K$ contains the critical path and is $\sigma$-invariant, it must contain the entire Hubbard tree by construction, and this is excluded for renormalization.

\begin{lemma}[Renormalizable subtree]
\label{Lem:renorm}
Let $\nu$ be a kneading sequence and suppose that $\Hub(\nu)$ contains a proper connected subtree $K_0$ that contains the critical value and at least one other point, and that satisfies $\sigma^p(K_0)\subset K_0$ for some $p\ge 2$. Then $\nu$ is $p$-renormalizable, except possibly if $\nu$ is $\star$-periodic with period dividing $p$.
\end{lemma}
\begin{proof}
We may assume that for $t \in \{1,\ldots  ,p-1\}$, the intersection $K_0\cap \sigma^t(K_0)$ is empty or consists of a single point (otherwise iterate this intersection until it contains the critical value and replace  $K_0$ with it; such an iterate exists because the intersection consists of more than one point and thus contains either a Fatou interval or a precritical point). This intersection point, if it exists, cannot be the critical value because the latter is an endpoint of $\Hub(\nu)$.

By $\sigma^p$-invariance, $K_0$ contains all points $\sigma^{kp}(\nu)$ for $k\ge 0$. Every interval of the form $[\sigma^{kp}\nu, \sigma^{(k+1)p}\nu]\subset K_0$ is mapped injectively at least $p-1$ times because otherwise an iterate $\sigma^t(K_0)$ with $t<p$ would contain the critical value. In other words,
\begin{equation*}
\sigma^p([\sigma^{kp}\nu, \sigma^{(k+1)p}\nu])\in \left\{ \left[\sigma^{(k+1)p}\nu, \sigma^{(k+2)p}\nu\right], \left[\nu, \sigma^{(k+1)p}\nu\right] \cup \left[\nu, \sigma^{(k+2)p}\nu\right]\right\}.
\end{equation*}
Let $K$ be the subtree spanned by the points $\sigma^{kp}\nu$ for $k\ge 0$. Then we conclude that $\sigma^p(K)\subset K$. So properties \eqref{Item:RenormStrictSubtree} and \eqref{Item:RenormInvariant} are satisfied,  \eqref{Item:RenormPeriodic} holds by assumption.
\end{proof}

\begin{lemma}[Intersection of little Hubbard tree]
\label{Lem:lht_cap} 
Let $\nu$ be $p$-renormalizable and let $K$ be its little Hubbard tree. For $t \in \{1,\ldots  ,p-1\}$, the intersection $K\cap \sigma^t(K)$ is empty or consists of a single point. 
\end{lemma}
\begin{proof}
If the statement does not hold, then $K\cap \sigma^t(K)$ contains an interval (possibly a Fatou interval) and thus a precritical point. Hence, there exists $s\ge 1$ such that $\sigma^s(K\cap \sigma^t(K))$ contains the critical value. Set
\begin{equation}
K_1:=\sigma^s(K\cap \sigma^t(K))\subset\sigma^s(K)\cap \sigma^{t+s}(K) 
\;.
\label{Eq:K_1}
\end{equation}
Since $K_1$ is also $p$-invariant, we have $K\subset K_1$. Let $r\in \{s,s+t\}$ be an integer that is not a multiple of $p$ (this is possible because $t$ is not divisible by $p$). Iterating then yields $K \subset K_1 \subset \sigma^r(K) \subset \sigma^{rp}(K) \subset  K$: the first inclusion was just stated, the second is in \eqref{Eq:K_1}, the third follows from $K\subset \sigma^r(K)$ by iteration, and the last follows from $\sigma^p(K)\subset K$. Together, this implies $K=\sigma^r(K)$. We claim that there is no point $\zeta \in K$ such that $\sigma^r(K)=\sigma^p \nu$, leading to a contradiction. First note that $\sigma^p\nu$ is an endpoint of the subtree $K$ (but not necessarily of the entire tree) and all such endpoints are of the form $\sigma^{kp}\nu$. So since $r$ is not a multiple of $p$, the point $\zeta$ is an interior point of $K$. Thus, a neighborhood of $\zeta$ can not map injectively to $\sigma^p \nu$. If $\zeta$ was a precritical point, then invariance under $\sigma^p$ would imply that a postcritical point $\sigma^{u}\nu$ with $u<p$ is an endpoint of $K$, which leads to a contradiction too.
\end{proof}

Non-recurrence and tree-finiteness are independent properties: both finite and infinite trees can be recurrent or not (see \cite[Section 4.2]{MalteDierkTrees} for examples).  The same does not hold for $\lambda$-finiteness as the following two lemmas point out (where we treat the non-periodic and periodic case separately).

\begin{lemma}[Non-periodic uniformly expanding sequences]
\label{Lem:renormalizable_nonperiodic} 
Every kneading sequence that is non-recurrent, non-periodic and non-renormalizable is uniformaly expanding. 
\end{lemma}
\begin{proof}
Since $\nu$ is non-recurrent and precritical points are dense, and since $\nu$ is an endpoint of $\Hub(\nu)$, there exists a precritical point $q\neq\nu$ such that $q\in[\nu,\nu^i]$ for every $i>0$; choose $q$ so that it has minimal depth, say $k$. We claim that there is a a number $\lambda$ such that $\sigma^\lambda([\nu,q])\supset [\star\nu,\nu]$.

The map $\sigma^k$ sends $[\nu,q]$ homeomorphically to $[\nu^k,\nu]\supset [q,\nu]$. 
Thus the sets $K_n=\sigma^{kn}([\nu,q])$ are monotonically increasing. If  $\star\nu\in K_n$ for some $n$, then the claim is proved.

The other case to consider is that $\star\nu\not \in K_n$ for all $n$. 
Let $K$ be the closure of $\bigcup_{n\ge 0}K_n$. This set is invariant under $\sigma^k$ by construction, it contains $\nu$ but not only $\nu$. This implies that either $\star\nu\not\in K$, or $\star\nu$ is an endpoint of $K$. But $\star\nu$ is never an endpoint of $\Hub(\nu)$ when $\nu$ contains an entry $\0$, which is always the case when $\nu$ is non-periodic. By Lemma~\ref{Lem:renorm} it follows that $\nu$ is renormalizable.

Now we show that $\nu$ is uniformly expanding with parameter $\lambda$. 
Let $[\rho,\rho']\subset\Hub(\nu)$ be an interval that is bounded by two precritical points of equal depth $n$. 
There must be a precritical point of depth less than $n$ on $[\rho,\rho']$.
Thus there is a postcritical point $\nu^t\neq\nu$ such that $\sigma^n([\rho,\rho'])\supset[\nu,\nu^t]\supset[\nu,q]$, hence $\sigma^{n+\lambda}([\rho,\rho'])\supset\sigma^\lambda([\nu,q])\supset [\star\nu,\nu]$. This proves the lemma.
\end{proof}

\begin{lemma}[Periodic uniformly expanding sequences]
\label{Lem:renormalizable_periodic} 
Every $\star$-periodic kneading sequence $\nu$ is uniformly expanding unless it is $q$-renormalizable such that $q$ is a proper divisor of the period of $\nu$. Stronger yet, there is a $\lambda'\in\N$ such that for every $\nu^k$ the image of $[\nu,\nu^k]$ after $\lambda'$ iterations contains the entire Hubbard tree. 
\end{lemma}
\begin{proof}
The proof is similar to the one before. Let $p$ be the period of $\nu$. If for all postcritical points $\nu^k$ with $k\in\{1,..,p-1\}$, there exists $i_k$ such that $\sigma^{i_k}([\nu,\nu^k])=\Hub(\nu)$, then we can take $\lambda(\nu)=\max_k i_k$: if $[\rho,\rho']$ is an interval bounded by two precritical points of equal depth $n$, then $\sigma^n([\rho,\rho'])\supset[\nu,\nu^k]$ for some $k>0$, and 
\[
\sigma^{n+\lambda}([\rho,\rho'])\supset \sigma^{\lambda-i_k}(\sigma^{i_k}([\nu,\nu^k]))\supset \sigma^{\lambda-i_k}(\Hub(\nu))=\Hub(\nu)
\;.
\]

The other case is that there exists $k\in\{1,..,p-1\}$ such that no iterate of $[\nu,\nu^k]$ ever covers the Hubbard tree. 
Since $\sigma^p([\nu,\nu^k])\supset[\nu,\nu^k]$, the sets $K_t:=\sigma^{tp}([\nu,\nu^k])$ are increasing and bounded by postcritical points, of which there are only finitely many, so there is a $K_T$ such that $K_{T+1}=K_T$, hence $\sigma^p(K_T)=K_T$. 

If any two of the $p-1$ sets $K_T,\sigma(K_T),\ldots  ,\sigma^{p-1}K_T$ intersect in more than a point, we look instead at the iterated images of the intersection, as argued in previous proofs. Thus, we find a set $K_0$ that is bounded by postcritical points and contains $\nu$ and at least another point, and so that $\sigma^p(K_0)=K_0)$ and $K_0$ intersects any of its $p-1$ iterates in at most one point.

Our next claim is that two of the sets $K_0,\sigma(K_0),\ldots  ,\sigma^{p-1}(K_0)$ are identical. This follows from the tree structure of $\Hub(\nu)$ and the fact that each of these $p-1$ sets contains at least two postsingular points, while the pairwise intersection contains at most one point.

Therefore there is an $s\in\{1,\ldots,p-1\}$ such that $\sigma^s(K_0)=K_0=\sigma^p(K_0)$. This implies that $\sigma^q(K_0)=K_0$ for $q:=\gcd(s,p)<s$. By Lemma~\ref{Lem:renorm}, we conclude that $\nu$ is $q$-renormalizable. 
\end{proof}

The previous two lemmas thus provide an upper bound on $N(n)$ by Lemma~\ref{Lem:UpperBoundEnya}. The next lemma shows the existence of a lower bound.

\begin{lemma}[Lower bound for periodic sequences]
\label{Lem:LowerBoundEnya}
Let $\nu$ be a non-renormalizable, $\star$-periodic kneading sequence. Then there exists $c>0$ such that $N(n)\ge c e^{h(\nu)n}$ for all $n$.
\end{lemma}
\begin{proof}
Divide the Hubbard tree of $\nu$ into $M$ edges $e_1,...,e_M$ formed by the critical point, postcritical points and branching points. The dynamics of the tree can then be described by a transition matrix $A$. As we mentioned previously, the core entropy $\nu$ equals the logarithm of the leading eigenvalue $\lambda_1$ of $A$. Let $\lambda_2,...,\lambda_M$ denote the other eigenvalues, counted with multiplicity. Since $\nu$ is non-renormalizable, Lemma~\ref{Lem:renormalizable_periodic} tells us that that there exists $t \in \N$ such that the $\sigma^t$ image of every of the $M$ edges equals the entire Hubbard tree. In other words, $A^t$ is a positive matrix, so $A$ is primitive. The Perron--Frobenius theorem then tells us that $|\lambda_i|< \lambda_1$ for $i=2,...,M$. By bringing $A$ into Jordan-normal-form we see that for every $j,k\in {1,...,M}$ there exists constants $c_1,...,c_M$ with $c_1>0$ such that for every $n$ 
\begin{equation}
\label{ineq:matrix}
(A^n)_{kj} \ge c_1 \lambda_1^n + c_2 n^M \lambda_2^n + ... c_M n^M \lambda_M^n. 
\end{equation}
Let $e_k$ be the edge on the critical path with endpoint $\star\nu$. Then $(A^n)_{kk}$ counts the number of precritical points of depth $n+1$ on $e_k$, which is less than $N(n+1)$, thus proving the claim.
\end{proof}

\Newpage

\section{H\"older continuity --- the non-renormalizable case}
\label{Sec:hoelder_nr}

Our first lemma on H\"older continuity has no prerequisite on $\nu$. 
\begin{lemma}[H\"older continuity from above]
\label{Lem:estimate_from_above}
Suppose that a kneading sequence $\nu$ with entropy $h=h(\nu)>0$ has a constant $C>0$ so that for all $n$ we have the bound $N_H(n)\le C e^{h n}$. Then for all $\nu' \succ \nu$  the entropy difference satisfies
\[
0\le h(\nu')-h(\nu) \le 2C  e^{-h k}
\]
where $k:=\diff(\nu,\nu')$.
\end{lemma}

\begin{proof}

We count the precritical points of generation $n$ on the critical path of $\Hub(\nu')$. Consider one such point and let $w$ be its itinerary; this is a finite word of length $n$. Then by Theorem~\ref{Thm:InclusionTrees} \eqref{item:combinatorics}
\begin{equation}
\label{rep1}
w=w_0e_0X_{l_0}w_1e_1X_{l_1}\ldots  X_{l_{s-1}}w_s
\end{equation}
where $s\ge 0$, $l_i\ge k-1$, $e_i\in \{\0,\1\}$,  $X_l$ stands for the first $l$ entries of $\nu$ and $w_i$ is a word from the Hubbard tree of $\mu$ where we allow $w_0e_0$ to be empty. Based on this representation, we are going to give an upper bound on the number of possible words of fixed length $n$. 

There are less than $\binom{n}{s}$ choices for the positions of the words $X_{l_0},...,X_{l_{s-1}}$ in the word $w$. For each choice, the lengths $|w_i|$ of the words $w_0,...,w_s$ and the positions of the $e_i$ are determined. For the $s$ entries $e_0$, \dots, $e_{s-1}$, we have $2^s$ choices. Moreover we have the bounds $N_H(|w_i|)\le Ce^{h|w_i|}$ and $\sum_{i=0}^s |w_i|\le n - ks$. 

These observations allow us to give an upper bound for $N_{\nu'}(n)$:

\begin{align*}
N_{\nu'}(n)
& \le \sum_{s= 0}^n \binom{n}s 2^s  \prod_{i=0}^s N_H(|w_i|)
\le  \sum_{s= 0}^n \binom{n}s 2^s C^{s+1} e^{h(n-ks)} \\
&= C e^{hn}  \sum_{s= 0}^n  \binom{n}s \left( 2C e^{-hk}\right)^s 
= C e^{hn}  \left( 1 + 2Ce^{-hk}\right)^n.
\end{align*}

This yields
\[
\frac 1 n \log N_{\nu'}(n) \le h + \log(1+2Ce^{-h k}) +\log C/n \le h + 2C e^{-h k}+\log C/n
\]
and the claim follows by taking the lim sup.
\end{proof}

The converse statement Theorem~\ref{Thm:InclusionTrees}~\eqref{item:combinatorics_reversed} yields lower bounds on the number of precritical points instead that show that the H\"older exponent is optimal:

\begin{theorem}[Optimal exponent for H\"older continuity]
\label{Thm:larger_expo}
Let $\nu$ be a $\star$-periodic, non-re\-nor\-ma\-li\-zable kneading sequence with positive core entropy.  Then core entropy is \emph{not} locally H\"older continuous with any exponent greater than $h(\nu)$. 
\end{theorem}
We should point out that requiring positive entropy is obsolete by Lemma~\ref{Lem:periodic_bound}. 
\begin{proof}

By Lemma~\ref{Lem:int_add}, there exists infinitely many sequences $\nu' \succ \nu$. For each of them, set $k:=\Diff(\nu', \nu)=\diff(\nu',\up\nu)$. 
Divide $\Hub(\nu)$ into edges formed by the critical value, postcritical points and all branching points and let $I$ be the edge on the critical value with endpoint $\star\nu$. Note that for every precritical point on $I$ there exists an analogue in $\sigma^{k-1}([\nu,\nu'])\cap [\star\nu', \nu]$ which does not leave $\Hub_\nu(\nu')$ before reaching the critical point by Theorem~\ref{Thm:InclusionTrees}~\eqref{it3}. By Theorem~\ref{Thm:InclusionTrees}~\eqref{item:combinatorics_reversed}, for every choice of $s\ge 0$ we have an injective map from $(s+1)$-tuples of precritical points $w_i\star\mu \in e$ to precritical points on the critical path of $\nu'$, which have the form
\[w_0e_0X_{l}w_1e_1X_{l}\ldots  X_{l}w_s\star\nu.\]
for some suitable $e_0, ..., e_{s-1} \in \{\0,\1\}$ and $l=k-1$. We now want to count those precritical points. For a word of length $n$ of the above form we claim that there are at least $\binom{n}{s}/2$ ways to choose the positions of the $s$ subwords $X_l$ when $n$ is large: There are $\binom{n}{s}$ ways to choose $s$ positions in a word of lenght $n$, however they need to be at least $s+1$ positions apart from each other. As $s$ is fixed and $n$ grows, this scenario becomes likely, showing the claim. Let $N^I(n)$ count the number of precritical points of depth $n$ on $I$. By Lemma~\ref{Lem:LowerBoundEnya}, or more precisely inequality $\eqref{ineq:matrix}$, we obtain $c>0$ such that $N^I(n)\ge c e^{h(\nu)n}$ for all $n$. A calculation like in the previous lemma yields:

\begin{align*}
N_{\nu'}(n)
& \ge \frac 1 2 \sum_{s= 0}^n \binom{n}s \prod_{i=0}^s N_I(|w_i|)
\ge \frac 1 2 \sum_{s= 0}^n \binom{n}s  c^{s+1} e^{h(\nu)(n-ks)} \\
&= \frac c 2 e^{h(\nu)n}  \sum_{s= 0}^n  \binom{n}s \left(c e^{-h(\nu)k}\right)^s 
= \frac c 2 e^{h(\nu)n}  \left( 1 + ce^{-h(\nu)k}\right)^n.
\end{align*}
In terms of entropy, this then implies $h(\nu')\ge h(\nu) + c e^{-h(\nu)k}$ and thus core entropy is not H\"older continuous for any exponent greater than $h(\nu)$. 
\end{proof}

\begin{lemma}[Characteristic branch point]
\label{Lem:mu}
Let $\nu$ be a non-recurrent kneading sequence that is not a bifurcation. Then there exists a characteristic periodic sequence $q$ that separates $\nu$ from its entire forward orbit, and there exists a natural number $r$ so that the following condition is satisfied:  for every (pre)periodic {characteristic} sequence $\mu$ on $[q,\nu]$ there exists a characteristic (pre)periodic sequence $\mu'$ with the following properties:
\begin{enumerate}
    \item \label{mu'_loc} $q \preceq \mu' \preceq \mu$;
    \item $q$ separates $\mu'$ from its entire forward orbit;
\label{mu'_iter} 
    \item \label{mu'_dist} 
$\diff(\mu',\nu) \ge \diff(\mu,\nu)-r-1$.
\end{enumerate}
\end{lemma}

\begin{proof}
On the critical path $[\star\nu,\nu]$ of $\Hub(\nu)$, let $b$ be the branch point closest to $\nu$; if there is no such branch point, let $b=\star\nu$. Non-recurrence implies that $\nu$ has a neighborhood without branch points (there are only finitely many precritical points of given depth, so branches that branch off from the critical path very close to $\nu$ must have endpoints that are postcritical and close to $\nu$). 
Therefore, $[b,\nu]$ is a non-trivial interval on the critical path without branch points on its interior.

Let $q_0:=\sup\{\nu^i\cap[\star\nu,\nu]\}$ and $q_1=\max\{q_0,b\}$. Since $\nu$ is non-recurrent, we have $q_1\neq\nu$. Let $r:=\diff(q_1,\nu)$ and let $\rho$ be the unique precritical point of depth $r$ on $[q_1,\nu]$. Finally, let $q$ be a characteristic periodic point on $[\rho,\nu]$ of minimal period, but at least period $r$. Then $\rho$ and $q$ have at least $r$ common entries. Since $\nu$ is not a bifurcation, there are infinitely many precritical points close to $\nu$, so $\rho$ and $q$ exist. 

Now we show that this choice of $q$ and $r$ makes it possible to choose for every $\mu$ a sequence $\mu'$ with the required properties.

If the orbit of $\mu$ never jumps behind $q$, simply choose $\mu'=\mu$ and all properties are satisfied. Otherwise, choose $m>0$ minimal such that $\sigma^m \mu\in[q,\nu]$.
The order on  $\Hub(\nu)$ between the points $\nu$, $\mu$, $\sigma^m \mu$, $q$, $\rho$, and $\nu^m$ is as indicated in Figure~\ref{Fig:RelativeOrderHubbNu} because $\rho$ is constructed so that it separates $\nu$ from all orbit points of $\nu$, while $\sigma^m \mu\in[\nu,q]$ by construction, and $\mu\in[\nu,\sigma^m\mu]$ because $\mu$ is characteristic.

Since $\rho\in[\sigma^m \mu,\nu^m]$, hence $\diff(\sigma^m\mu,\nu^m)\le r$, we have
\begin{equation}
\label{eq:mr}
\diff(\nu, \mu)\le m+r.
\end{equation}

We will distinguish three cases. In all of these our point $\mu'$ will be a preperiodic preimage of $q$. The three conditions \eqref{mu'_loc}--\eqref{mu'_dist} that $\mu'$ should satisfy all have geometric interpretations: the first describes an interval where $\mu'$ should be located, the third requires that $\mu'$ should be sufficiently close to $\mu$, and the seond implies that $\mu'$ cannot be too close to $\mu$ because otherwise its $m$-th iterate would violate condition~\eqref{mu'_iter}.

\medskip 
\emph{Case 1: the map $\sigma^{m}$ is injective on $[\mu,q]$}. In this case, we have a homeomorphism $\sigma^{ m}\colon[\mu,q]\to[\sigma^m \mu,\sigma^m q]\ni q$, so there is a unique preimage $\mu':=(\sigma^{ m})^{-1}(q)$ on $[\mu,q]$. It satisfies condition \eqref{mu'_loc} by construction, and it satisfies condition \eqref{mu'_iter}  because the first $m-1$ iterates are obvioulsy before $q$, and from then on the claim follows because $q$ is characteristic. Condition \eqref{mu'_dist} follows from \eqref{eq:mr} and $\diff(\mu, \mu')\ge m$.
\medskip

\emph{Case 2: the maximal number of injective iterations on $[\mu,q]$ is $k\le m-2$}. Then $\sigma^{ k}\colon[\mu,q]\to[\sigma^k \mu,\sigma^k q]\ni\star\nu$ and there is a unique precritical point $\zeta\in[\mu,q]$ of depth $k+1$, so $\sigma^{(k+1)}\colon[\mu,\zeta]\to [\sigma^{k+1}\mu,\nu]$ is a homeomorphism. Since $k+1<m$, we have $q\in[\sigma^{k+1} \mu,\nu]$. There is a unique $\mu'\in[\mu,\zeta]$ with $\sigma^{(k+1)}(\mu')=q$, and this choice satisfies again conditions \eqref{mu'_loc} and \eqref{mu'_iter} for the same reasons as in case 1. In this case our choice of $\mu'$ is not closest possible to $\mu$, so if $\diff(\mu, \mu')\ge m-1$ does not hold, simply repeat the argument for $\zeta$ in place of $q$ (note that we have not used periodicity of $q$ in the proof).

\medskip

\emph{Case 3: the maximal number of injective iterations on $[\mu,q]$ is $k= m-1$}. In this case we again have $\sigma^{ k}\colon[\mu,q]\to[\sigma^k \mu,\sigma^k q]\ni\star$ and there is a unique precritical point $\zeta\in[\mu,q]$ of depth $k+1=m$. Now we consider the homeomorphism $\sigma^{m}\colon[\zeta,q]\to [\nu,\sigma^m q]\ni q$ and the unique point $\mu'\in[\zeta,q]$ with $\sigma^{m}(\mu')=q$. It again satisfies conditions \eqref{mu'_loc} and \eqref{mu'_iter} by construction, and condition \eqref{mu'_dist}   follows from $\diff(\mu, \mu')=m-1$. 
\end{proof}

\begin{figure}
\begin{picture}(280,20)
\put(3,13){\line(1,0){280}}
\put(0,3){$\nu$}
\put(76,3){$\mu$}
\put(135,3){$\sigma^m \mu$}
\put(177,3){$q$}
\put(207,3){$\rho$}
\put(235,3){$\nu^m$}
\put(3,13){\circle*{4}}
\put(80,13){\circle*{4}}
\put(140,13){\circle*{4}}
\put(180,13){\circle*{4}}
\put(210,13){\circle*{4}}
\put(240,13){\circle*{4}}
\end{picture}
\caption{Relative position of points of Lemma \ref{Lem:mu} (the order increases to the left). The points $\zeta$ and $\mu'$ lie between $\mu$ and $q$, but possibly at either side of $\sigma^m \mu$. }
\label{Fig:RelativeOrderHubbNu}
\end{figure}
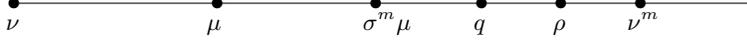

\begin{theorem}[H\"older continuity, non-renormalizable case]
\label{Thm:HoelderNonRenorm}
Let $\nu$ be non-recurrent, non-renormalizable, and tree-finite,  and such that $h=h(\nu)>0$. Then for every $\eps>0$ there is a neighborhood $U$ of $\nu$ such that all $\nu'\in U$ satisfy 
\[
|h(\nu)-h(\nu')|\le  e^{-(h-\eps) k}
\]
where $k=\Diff(\nu,\nu')$. 
\label{Thm:Hoelder_continuity}
\end{theorem}

\begin{remark}
We may drop locality and say that the statement holds for all $\nu'$, at the expense of introducing a multiplicative constant $C_\nu=e^{hK}\log 2$ where $K=\max_{\nu'\not \in U} \Diff(\nu,\nu')$. 
\end{remark}

\begin{proof}
Using non-recurrence, choose $q\prec\nu$ and $r$ according to Lemma~\ref{Lem:mu}. Choose a periodic, characteristic, non-renormalizable sequence  $q'\in(q,\nu)$ (remember that periodic, characteristic sequences exist arbitrarily close to $\nu$ by Lemma~\ref{Lem:int_add}; since $\nu$ is non-renormalizable, there are sufficiently nearby sequences that are not renormalizable either). Let $L:=\Diff(q',\nu)$.

We want to show that $\nu$ is uniformly expanding with a parameter that is dynamically controlled by $q$ and $q'$. By Lemma~\ref{Lem:renormalizable_periodic}, the kneading sequence $q'$ is uniformly expanding in the stronger sense that there is a parameter $Q_0>0$ such that $\sigma^{Q_0}([q',\sigma^k(q')])=\Hub(q')$ for every $k$ with $\sigma^k(q')\neq q'$. Since the interval $[q',q]\subset\Hub(q')$ contains a precritical point of $q'$, we then find $Q_1\ge Q_0$ such that $\sigma^{Q_1}([q',q])=\Hub(q')$. 

Then by Theorem~\ref{Thm:InclusionTrees}, the interval $[q,q']$ can be refound in $\Hub(\nu)$ such that $\sigma^{Q_1}([q,q'])\supset \Hub_{q'}(\nu)\ni\star\nu$. Thus after $Q:=Q_1+1$ iterations it covers the critical path $[\star\nu, \nu]$. 

By the choice of $q$, we have $[\nu,\nu^i]\supset [q',q]$ for any iterate $\nu^i$ of $\nu$ distinct of $\nu$. Hence, the $Q$-th iterate of $[\nu,\nu^i]$ must contain the entire critical path. So indeed, $\nu$ is uniformly expanding with parameter $Q$. 

Suppose that $\eps<h(\nu)$. 
Let $U$ be a neighborhood of $\nu$ such that all $\nu'\in U$ have $|h(\nu')-h(\nu)|<\eps$; by continuity of entropy \cite{DimaEntropy,TiozzoContinuity}, there is indeed such a neighborhood. Restrict $U$ if necessary so that all $\nu'\in U$ satisfy $\Diff(\nu',\nu)>2(L+r)$.

Now consider a particular $\nu'\in U$ and set $k:=\Diff(\nu,\nu')$. Let $\mu$ be the periodic point from Theorem~\ref{Thm:branch_weak} which satisfies $\mu\preceq \nu$ and $\mu \preceq \nu'$. That theorem yields
\[
\Diff(\nu,\mu)\ge \Diff(\nu,\nu')=k
>2(L+r)>L=\Diff(q',\nu)\;.
\] 
Since $\nu$, $\mu$, $q'$, and $q$ are linearly ordered, this implies  $\mu\succ q'\succ q$.  Lemma~\ref{Lem:mu} provides a sequence $\mu'$ with $q \prec \mu'\prec\mu$ such that $q$ separates $\mu'$ from its forward orbit and such that
\[
 \Diff(\mu',\nu) \ge \Diff(\mu,\nu)-(r+1)\ge k-(r+1) \ge 2L+r-1 >L
\;,
\] 
and this implies $\mu'\succ q'$. 

Similar as for $\nu$, the sequence $\mu'$ is uniformly expanding with parameter $Q$ because of Theorem~\ref{Thm:InclusionTrees} and the fact $[\mu',(\mu')^i]\supset [q',q]$, where $(\mu')^i$ denotes any iterate of $\mu'$ distinct from $\mu'$. 

Since $\mu'\prec\nu$, we have $\kappa(\mu')\le\kappa(\nu)<\infty$ (Lemma~\ref{Lem:KappaMonotone}). Therefore, Lemma~\ref{Lem:UpperBoundEnya} implies that in the dynamics of $\Hub(\mu')$, we have 
\[
N_H(n) \le Ce^{h(\mu')n} \text{ with } C= \kappa(\mu') e^{h(\mu')(Q+\kappa(\mu'))}\le  \kappa(\nu) e^{h(Q+\kappa(\nu))}
\;. 
\]

We can thus apply Lemma~\ref{Lem:estimate_from_above} to $\nu\succ\mu'$  and obtain 
\[
0\le h(\nu)-h(\mu')\le 2C e^{-h(\mu')\Diff(\nu,\mu')} 
\le 2C e^{-h(\mu')(k-(r+1))} 
\;.
\]
The analogous bound holds for the entropy difference between $\nu'\succ\mu'$, using $\Diff(\mu',\nu')\ge \min\{\Diff(\mu',\nu),\Diff(\nu,\nu'\}\ge \min\{k-(r+1),k\}=k-r-1$. 
We thus have
\[
|h(\nu)-h(\nu')| \le  2C e^{-h(\mu')(k-(r+1))} =  2Ce^{h(\mu')(r+1)} e^{-h(\mu')k}
\;.
\]
By construction, we have $h(\nu)\ge h(\mu')\ge h(\nu)-\eps$, so with
\[
C_\nu=  2Ce^{h(\mu')(r+1)}\le 2\kappa(\nu) e^{h(Q+\kappa(\nu))} e^{h(r+1)} = 2\kappa(\nu) e^{h(Q+\kappa(\nu)+(r+1))}
\]
we have 
\begin{equation}
\label{eq:hoelderwithC}
|h(\nu)-h(\nu')|\le C_\nu  e^{-(h-\eps) k}
\;.
\end{equation}
To get rid of the constant $C_\nu$ in \eqref{eq:hoelderwithC}, one adjusts $\eps$ and shrinks the neighborhood $U$ by enlarging $k$ such that $C_\nu e^{-\eps k}<1$; this gives

\[
|h(\nu)-h(\nu')|\le C_\nu  e^{-(h-2\eps+\eps) k} = C_\nu e^{-\eps k} e^{-(h-2\eps)k}<e^{-(h-2\eps) k}
\;.
\]
\end{proof}

\section{H\"older continuity --- the renormalizable case}
\label{Sec:hoelder_r}

We start this section with a characterization of renormalizable kneading sequences.

\begin{lemma}[Renormalizable kneading sequence]
\label{Lem:LittleHubbardTree}
A kneading sequence $\nu$ is $p$-renormalizable (with $p\in\N$, $p\ge 2$)  if and only if there exists a characteristic periodic sequence $\mu_\circ\in\{\0,\1\}^\N$ of period $p$ such that $\diff(\sigma^{kp}(\nu),\mu_\circ)\ge p$ for all $k\ge 0$, but  $\sigma^{p}(\nu)\neq \nu$.          
\end{lemma}

The condition $\diff(\sigma^{kp}(\nu),\mu_\circ)\ge p$ for all $k\ge 0$ means that the sequence $\nu$ differs from the periodic sequence $\nu_*$ only at positions that are multiples of $p$.

We thus have a $\star$-periodic kneading sequence $\mu$ such that $\mu_\circ=\up\mu$ or $\mu_\circ=\lo\mu$. The proof will imply that the latter case can only apply if $\nu$ is not complex admissible. 

Clearly, $\mu$ is uniquely determined by $\nu$ and $p$ (if it exists). We call $\mu$ the \emph{base sequence} of the renormalization and $\mu_\circ$ its \emph{dynamical sequence}. The base sequence may itself be renormalizable (with period a strict divisor of $p$). Therefore, there is a minimal period $p\ge 2$ such that $\nu$ is $p$-renormalizable. The base sequence for this minimal period is called the \emph{maximal base sequence}; then $\mu$ itself is not renormalizable. 

\begin{proof}

Suppose $\nu$ is $p$-renormalizable and let $K$ be its little Hubbard tree. Since $K$ contains more than one point, we clearly have $\sigma^{p}(\nu)\neq \nu$. Moreover, $\sigma^p(K)\subset K$ implies that the itinerary of $\nu$ is periodic of period $p$, except that the entries at positions that are multiples of $p$ may be arbitrary. Let $\mu$ thus be the unique $\star$-periodic sequence of period $p$ satisfying $\Diff(\nu, \mu)\ge p$. If the internal address of $\mu$ is an initial segment of the internal address of $\nu$, then Lemma~\ref{Lem:int_add} implies $\mu \prec \nu$ and $\mu_\circ=\up\mu$ satisfies $\diff(\sigma^{kp}(\nu),\mu_\circ)\ge p$. Otherwise, the first $p$ entries of $\nu$ coincide with $\lo\mu$. Then the internal address of a non-standard bifurcation $\eta$ of $\mu$ must be contained in the internal address of $\nu$. By Lemma~\ref{Lem:int_add}, we thus have $\eta \preceq \nu$ and therefore the sequence $\mu_\circ=\lo\mu$ is characteristic in $\Hub(\nu)$. It clearly satisfies $\diff(\sigma^{kp}(\nu),\mu_\circ)\ge p$. 

For the converse, suppose $\Hub(\nu)$ contains a $p$-periodic sequence $\mu\prec\nu$ such that 
$
\diff(\sigma^{kp}(\nu),\mu_\circ)\ge p  $  for all $k\ge 0$ and $\sigma^p(\nu)\neq \nu$. Let $K$ be the connected hull of the points $\nu^{kp}$ for $k\ge 0$. By assumption,
$K$ contains more than one point, and $\nu$ and $\nu^{kp}$ coincide for at least $p-1$ entries. So $\sigma^p$ maps the interval $[\nu,\nu^{kp}]$ onto $[\nu^p,\nu^{2p}]$ or onto $[\nu,\nu^p]\cup [\nu,\nu^{2p}]\subset K$. So indeed we have $\sigma^p(K)= K$ and therefore $\nu$ is $p$-renormalizable by Lemma~\ref{Lem:renorm}.\end{proof}

\begin{definition}[Little Mandelbrot set]
The \emph{little Mandelbrot set} $\LM(\mu)$ of a $\star$-periodic kneading sequence $\mu$ is defined as the set of renormalizable kneading sequences $\nu$ with base sequence $\mu$. 
\end{definition}

\begin{lemma}[Little Mandelbrot set connected]
\label{Lem:lilmandel}
If $\mu \prec \eta \prec \nu$ and $\nu \in \LM(\mu)$, then $\eta \in \LM(\mu)$.
\end{lemma}
\begin{proof}
By Lemma~\ref{Lem:LittleHubbardTree}, all entries of $\nu$ that are not at a position that is a multiple of the period of $\mu$, coincide with $\mu$. Recursively, one can see that the same property holds for all precritical points on $[\mu, \nu]\subset \Hub(\nu)$. The density properties of precritical points imply that this property holds for all itineraries on this interval, in particular for $\eta$. Hence, $\eta$ has $\mu$ as base sequence of renormalization by Lemma~\ref{Lem:LittleHubbardTree}.
\end{proof}

Lemma~\ref{Lem:LittleHubbardTree} gives rise to the following definition:

\begin{definition}[De-renormalized sequence]
For a $p$-renormalizable sequence $\nu$ define a map $\rho_p\colon \{\0,\1,\star\}^\infty \to \{\0,\1,\star\}^\infty$ such that $s=s_1s_2s_3\dots$ maps to $\rho_p(s)=s_ps_{2p}s_{3p}\dots$. We call $\rho_p(\nu)$ the \emph{de-renormalized sequence} of $\nu$ (with respect to $p$-renormalization).
\end{definition}

In other words, $\rho_p(s)$ is the sub-sequence with positions that are multiples of $p$: these are exactly the entries that are not fixed by the renormalization. 

\hide{
Let $\nu$ be a $p$-renormalizable kneading sequence and $\rho_p\colon \{\0,\1,\star\}^\infty \to \{\0,\1,\star\}^\infty$ be the map on sequences $s=s_1s_2s_3\dots$ to the sequences $\rho_p(s)=s_ps_{2p}s_{3p}\dots$ taking only the entries not predetermined by the base sequence. We then call $\rho_p(\nu)$ the \emph{de-renormalized sequence} of $\nu$ (with respect to $p$-renormalization).
}

The following lemma explains why we speak of a little Mandelbrot set.

\begin{lemma}[Self-similarity of renormalization]
\label{Lem:fractal}
Let $\nu$ and $ \{\nu_k\}_{k\ge 1}$ be renormalizable kneading sequences, all with the same base sequence $\mu$ of period $p$. Let $\eta$ and $\{\eta_k\}_{k\ge 1}$ be the corresponding de-renormalized sequences and let $K$ and $\{K_k\}_{k\ge 1}$ be the corresponding little Hubbard trees. Then we have the following:
\begin{enumerate}
\item \label{fractal} $\rho_p(K)=\Hub(\eta)$\;;
\item \label{dorder} $\nu_1 \prec \nu_2$ if and only if $\eta_1 \prec \eta_2$\;;
\item \label{dconv} $\nu_k \to \nu$ if and only if $\eta_k \to \eta$.
\end{enumerate}
\end{lemma}
\begin{proof}
\eqref{dconv} is trivial and \eqref{fractal} follows directly from the fact that all itineraries in $K$ coincide except at every $p$-th entry. For \eqref{dorder}, if $\nu_1 \prec \nu_2$, then $K_1$ can be found inside $K_2$ and so $\Hub(\eta_1)$ is found inside $\Hub(\eta_2)$ with $\eta_1$ being characteristic. Conversely, if $\eta_1 \prec \eta_2$, then $\nu_1 \in K_2$ and no iterated $\sigma^p$-image of $\nu_1$ maps behind it. Since the images of the little Hubbard tree intersect only at images of $\mu$ or not at all, no other iterates of $\nu_1$ map behind it either, thus $\nu_1$ is characteristic in $\Hub(\nu_2)$. 
\end{proof}

Parts of the following results are mild extensions and improvements of \cite[Lem\-ma~6.20]{DimaEntropy}, showing that entropy is constant in little Mandelbrot sets of positive entropy. Note that here, unlike in \cite{DimaEntropy}, we include kneading sequences that are not complex admissible, so our ``little Mandelbrot sets'' are not necessarily subsets of the standard Mandelbrot set (not even in a combinatorial sense).

\begin{lemma}[Entropy, lower bound for periodic sequence]
\label{Lem:periodic_bound}
If $\nu$ is $\star$-periodic with some period $p$ and $h(\nu)>0$, then $h(\nu)\ge \log 2/(p-1)$ and $N(n)\ge 2^{[n/(p-1)]}$ for infinitely many $n$ (even for all $n$ if $\nu$ is not renormalizable).
\end{lemma}

\begin{remark}
It is well known that the $\star$-periodic kneading sequences with zero entropy are exactly the ones constructed by repeated bifurcations of the trivial kneading sequence (see for example \cite[Proposition~2.11]{HenkBiaccDim}). 
\end{remark}

\begin{proof}
First suppose that the point $-\alpha$ is contained in the Hubbard tree $T:=T(\nu)$, so there is a postcritical point $\nu_j$ with $-\alpha\in[\alpha,\nu^j]$. Since $\nu$ has period $p$, we clearly have $j\le p-2$.

Then $\sigma$ sends $[\alpha,-\alpha]$ two-to-one to $[\alpha,\nu]$, and $\sigma^{j}$ sends $[\alpha,\nu]$ to its image that must contain $[\alpha,\nu^j]$, homeomorphic or not. Thus $\sigma^{j+1}$ has a complete $2$-horseshoe on $[\alpha,-\alpha]$, which implies $h(\sigma^{j+1})\ge \log 2$ and thus $h(\sigma)\ge \log 2/(j+1)\ge\log 2/(p-1)$. We even obtain $N(n)\ge 2^{[n/j]}>2^{[n/(p-1)]}$.

Now suppose $-\alpha \not\in \Hub(\nu)$. Let $b$ denote the number of edges branching off at $\alpha$  and let $K_0$ be the connected component of $\Hub(\nu)\setminus \alpha$ containing the critical point. When iterating $K_0$, we can only land in one component branching off at $\alpha$ because otherwise we would have found a preimage of $\alpha$ different from $\alpha$. Hence, the iterates of  $K_0$ cycle along the branches and $\sigma^b(K_0)\subseteq K_0$. Thus, $\nu$ is renormalizable. If $\nu$ is a bifurcation of a sequence $\nu'$, replace $\nu$ with $\nu'$. It is known that bifurcations have the same entropy, so the bounds only get better for $\nu$ if we show them for $\nu'$. Since $\nu$ has positive entropy, repeating this process does not terminate at the trivial kneading sequence. 

Choose $q$ maximal such that $\nu$ is $q$-renormalizable. Let $\eta$ be the de-renormalized sequence of period $q':=p/q$. Let $N'(n)$ count the precritical points of depth $n$ on the critical path of $\eta$. By Lemma~\ref{Lem:fractal}\eqref{fractal} we know $N(qn)\ge N'(n)$. Since $q$ was chosen maximal, the sequence $\eta$ is not renormalizable. The internal address of $\eta$ is not of the form $1\to b$ either because we assumed that $\nu$ is not a bifurcation. Hence, $\eta$ has positive entropy  (see remark above) and by the first part of the proof, $N'(n)\ge 2^{[n/(q'-1)]}$. Hence, $N(n)\ge 2^{[n/(q(q'-1))]}\ge 2^{[n/(p-1)]}$. 
\end{proof}

\begin{lemma}[Entropy on $\LM(\mu)$ when $h(\mu)=0$]
\label{Lem:lilmandel_entropy_zero}
Let $\nu$ be $p$-renormalizable with positive core entropy but such that the base sequence $\mu$ has entropy zero. Let $\eta$ be the associated de-renormalized sequence and assume $\eta$ is non-renormalizable and tree-finite. Then $h(\nu)=h(\eta)/p$ and for some growth constant $C>0$
\begin{equation}
\label{ineq:denorm}
N(n)\le C e^{(h(\eta)/p)n}.
\end{equation}
The  constant $C$ only depends on $h(\eta)$, $p$ and an analogous growth constant of $N_\eta$. 
\end{lemma}
\begin{proof}
Let $N_K(n)$ count the precritical points on the little Hubbard tree $K$ and let $N'(n)$ count the precritical points on $\Hub(\eta)$. Then we have $N_K(pn)=N'(n)$ and $N_K(m)=0$ if $m$ is not a multiple of $p$. Since $\eta$ is tree-finite, so is $\nu$ and thus $h^H(\nu)=h(\nu)\ge h(\eta)/p$. Since $\eta$ is also non-renormalizable, a bound for $N_K(n)$ analogous to \eqref{ineq:denorm} holds by Lemmas~ \ref{Lem:UpperBoundEnya}, \ref{Lem:renormalizable_nonperiodic} and \ref{Lem:renormalizable_periodic}.

Set $I=[\mu, \nu]\setminus K$. Then the interval $I$ is expanding under $\sigma^p$ and $\sigma^p I \subseteq I\cup K$. Let $N_*(n)$ count the precritical points on $[\mu, \nu]$. Then we find
\[
N_*(n)\le N_K(n)+N_K(n-p)+N_K(n-2p)+\dots,
\]
and thus \eqref{ineq:denorm} holds for $N_*(n)$ and a modified constant, say $C_*$. Since $h(\mu)=0$, we have a trivial bound of the form $N_{\mu}(n)\le C_\mu e^{(h(\eta)/2p)n}$ for some constant $C_\mu$. Any precritical point on the critical path of $\nu$ will coincide with a precritical point of $\mu$ until they land in $[\mu, \nu]$, hence
\begin{align*}
N(n) &\le \sum_{k=0}^n 2N_{\mu}(k)N_*(n-k) \le \sum_{k=0}^n 2C_\mu C_*e^{(h(\eta)/2p)k}e^{(h(\eta)/p)(n-k)} \\
&< \left( 2C_\mu C_* \sum_{k=0}^\infty e^{-(h(\eta)/2p)k} \right) e^{(h(\eta)/p)n}=: C e^{(h(\eta)/p)n}. 
\end{align*}
Since we have already established the lower bound, this implies $h(\nu)=h(\eta)/p$. 
\end{proof}

\begin{lemma}[Entropy on $\LM(\mu)$ when $h(\mu)>0$]
\label{Lem:lilmandel_entropy}
Suppose a base sequence $\mu$ has positive entropy. Then entropy of all kneading sequences within $\LM(\mu)$ is constant and coincides with $h(\mu)$.

Moreover, for all $\nu\in \LM(\mu)$ we have the upper bound
\[
N_\nu^H(n)\le C_\mu e^{h(\mu)n}
\]
where the constant $C_\mu$ only depends on $\mu$.
\end{lemma}
\begin{proof}
If the base sequence $\mu$ is renormalizable such that the base sequence of this renormalization, say $\mu'$, also has positive entropy, then replace $\mu$ by $\mu'$. Repeat this argument as long as possible; it must terminate because the periods decrease in every step. Therefore, we may assume that either $\mu$ is non-renormalizable, or that $\mu$ is renormalizable and none of its base sequences have positive entropy.

Let $p$ be the period of $\mu$. We call a precritical point $\rho\in\Hub(\nu)$ \emph{renormalizable} if there exists $t\in \{0,...,p-1\}$ such that the sequences of $\rho$ and $\nu^t$ coincide except at positions $kp-t$ for $k\in\N$. Denoting the number of renormalizable precritical points of depth $n$ on $\Hub(\nu)$ by $N_r^H(n)$, we thus have $N_r^H(n)\le p 2^{n/p}$.  

If $\mu$ is non-renormalizable, then $\mu$ is uniformly expanding by Lemma~\ref{Lem:renormalizable_periodic}. Thus, the upper bound $N_{\mu}^H(n)\le C e^{h(\mu) n}$ follows from Lemma \ref{Lem:UpperBoundEnya}. In the other case, we obtain the same bound from Lemma~\ref{Lem:lilmandel_entropy_zero} instead (choosing the maximal period of renormalization for $\mu$ ensures that the de-renormalized sequence is itself not renormalizable). 

Now let $\rho$ be any precritical point on $\Hub(\nu)$. If $\rho$ never leaves the subtree spanned by the orbit of the dynamical sequence $\mu_\circ$, it corresponds to a precritical point of $\Hub(\mu)$ by Theorem~\ref{Thm:InclusionTrees}. On the other hand, Lemma~\ref{Lem:lilmandel} implies that all precritical points outside this subtree are renormalizable. Similarly as in the proof of Theorem~\ref{Thm:InclusionTrees}~\eqref{item:combinatorics}, we thus have that every precritical point on $\Hub(\nu)$ is of the form $w_\mu e w_r\star\nu$, where $w_\mu$ is a word of a precritical point of $\Hub(\mu)$, the symbol $e\in\{\0,\1\}$, and $w_r\star\nu$ is renormalizable. Organizing the count by the number $k=|w_r|+1$ such that the precritical point becomes renormalizable after $n-k$ iterations, we can compute

\begin{align*}
N_\nu^H(n) &\le \sum_{k=0}^n 2N^H_\mu(n-k)N^H_r(k) \le \sum_{k=0}^n 2C e^{h(\mu)(n-k)}p e^{(\log 2/p)k}\\&< \left(2Cp \sum_{k=0}^\infty e^{-(h(\mu)-\log 2/p)k}\right)e^{h(\mu)n} =: C_\mu e^{h(\mu)n} 
\;.
\end{align*}
The geometric series converges by Lemma~\ref{Lem:periodic_bound}, so we obtain the desired upper bound. This implies $h(\nu)\le h(\mu)$, while monotonicity of entropy (as consequence of Theorem~\ref{Thm:InclusionTrees}) yields $h(\mu)\le h(\nu)$, so we have equality.
\end{proof}

In particular, we have shown the following.

\begin{corollary}[Entropy identity]
Let $\nu$ be renormalizable with base sequence $\mu$ of period $p$ and de-renormalized sequence $\eta$. Then
\[
h(\nu)=\max(h(\mu), h(\eta)/p).
\]
\end{corollary}

\begin{theorem}[H\"older continuity, renormalizable case]
\label{Thm:HoelderRenorm}
Let $\nu$ be renormalizable so that one of its base sequence has positive core entropy. Then there is a neighborhood $U$ of $\nu$ such that all $\nu'\in U$ satisfy 
\begin{equation*}
|h(\nu)-h(\nu')|\le C e^{-h(\nu) k}
\end{equation*}
where $k=\Diff(\nu,\nu')$, and $C>0$ is a constant that depends only on $\nu$. \newline
\end{theorem}
Note that $\nu$ is allowed to be recurrent and tree-infinite and its maximal base sequence may still have zero entropy. Moreover, just as in the remark after Theorem~\ref{Thm:Hoelder_continuity},  locality can be dropped by enlarging the constant. Observe also that in this theorem, the H\"older exponent is $h(\nu)$, not up to an $\eps$. This implies that we cannot get rid of the constant as we did before.
 
\begin{proof}
We can choose a base sequence $\mu$ that is maximal in the sense that it itself is not renormalizable or all of its base sequences have zero entropy while $h(\mu)>0$. 

Let $p$ be the period of $\mu$ and let $U$ be the neighborhood consisting of 
all kneading sequences $\nu'$ with $\Diff(\nu,\nu')> 2p$. Now consider $\nu'\in U$ and let $\tau$ be the (weak) branch point of $\nu$ and $\nu'$ from Theorem~\ref{Thm:branch_weak} 
(where we employ the convention $\tau=\nu$ if $\nu \prec \nu'$ and $\tau=\nu'$ if $\nu' \prec \nu$). By Lemma~\ref{Lem:lilmandel} and $\Diff(\tau, \nu)\ge \Diff(\mu,\nu)$, we have $\tau \in \LM(\mu)$.  

By Lemma~\ref{Lem:lilmandel_entropy}, we have $h(\tau)=h(\mu)=h(\nu)$ and $N_\tau^H(n)\le C_\mu e^{h(\mu)n}=C_\mu e^{h(\tau)n}$, so we can apply Lemma \ref{Lem:estimate_from_above} and obtain
\[
0\le  h(\nu')-h(\tau) \le 2 C_\mu e^{-h(\tau)\Diff(\tau,\nu')}
\;.
\]
The claim follows because $h(\tau)=h(\nu)$ and $k=\Diff(\nu,\nu')\le \Diff(\tau,\nu')$.
\end{proof}

\begin{theorem}[H\"older continuity, renormalizable case, close to zero entropy]
\label{Thm:Hoelder_low}
Let $\nu$ be non-recurrent, tree-finite and renormalizable with $h:=h(\nu)>0$ but such that all base sequences have entropy zero. Then for every $\eps>0$ there exists a neighborhood $U$ of $\nu$ such that all $\nu' \in U$ satisfy
\[
|h(\nu)-h(\nu')|\le e^{-(h-\eps)k}.
\]
where $k=\Diff(\nu, \nu')$. 
\end{theorem}
\begin{proof}
The proof follows the main steps as for Theorem~\ref{Thm:HoelderNonRenorm}. Just like there, it suffices to find constants $C_\nu, R$ and a neighborhood $U$ only depending on $\nu$ and $\epsilon$ such that for any $\nu' \in U$ we can find $\tau' \in U$ such that
\begin{enumerate}
\item $\tau'\prec \nu$, $\tau'\prec \nu'$ and $h(\tau')\ge h(\nu)-\eps$
\item \label{it:taubound} $N_{\tau'}(n)\le C_\nu e^{h(\tau')n}$ and
\item $\min(\Diff(\tau',\nu),\Diff(\tau', \nu))\ge \Diff(\nu, \nu')-R$. 
\end{enumerate}
We claim that there exists $p$ maximal such that $\nu$ is $p$-renormalizable. If not, we would have a sequence of base sequences converging to $\nu$, but since they are all assumed to have zero entropy, this contradicts continuity of entropy. 
Let $\mu$ thus be the $p$-periodic base sequence of $\nu$ and $\eta$ the associated de-renormalized sequence. By maximality, the sequence $\eta$ is not renormalizable. By Lemma~\ref{Lem:fractal}~\eqref{fractal}, it inherits non-recurrence and tree-finiteness from $\nu$. 

Given $\nu'$, let $\tau$ be its branch point with $\nu$ (again in the sense of Theorem~\ref{Thm:branch_weak}). If $\nu'$ is close enough to $\nu$, we can ensure that $\tau$ is also renormalizable with base sequence $\mu$. Let us now apply Lemma~\ref{Lem:mu} with $\eta$ and $\rho_p(\tau)$ to yield a $\star$-periodic sequence $\mu'$ with the properties listed in the lemma. Hence, we find a constant $C_\nu'$ (that does not depend on $\nu'$) such that $N_{\mu'}(n)\le C_\nu' e^{h(\mu')n}$. Let $\tau'$ be the renormalizable sequence with base sequence $\mu$ such that $\rho_p(\tau')=\mu'$. By  Lemma~\ref{Lem:fractal}~\eqref{dorder}, we have $\mu \prec \tau' \preceq \tau \preceq \nu$. By Lemma~\ref{Lem:lilmandel_entropy_zero}, we then have \eqref{it:taubound} as the constant can be computed from $C_\nu', p, \eps$ and $h(\nu)$. The other properties also follow as their equivalences hold for $\mu'$ after adjusting $U$. 
\end{proof}

We conclude this section by discussing kneading sequences with entropy zero: here the conjecture is that one does not have H\"older continuity (for any positive exponent). We confirm this conjecture for the Feigenbaum sequence; there is a similar argument in greater generality.

\begin{corollary}[No H\"older continuity at zero entropy]
\label{Cor:zero_entropy}
There exists a kneading sequence $\nu$ with $h(\nu)=0$ such that core entropy is \emph{not} locally H\"older continuous at $\nu$ for any exponent. 
\end{corollary}
\begin{proof}
Consider $\nu$ to be the so-called Feigenbaum point with internal address $1-2-4-8-16-\ldots  $. Note that all the approximating periodic sequences with internal addresses $1-2,1-2-4,1-2-4-8,\ldots  $ are bifurcations of $\overline{\star}$. As such, precritical points only grow linearly and they have entropy zero. By continuity of entropy, $\nu$ has zero entropy, too. Next, we want to look at the sequence $\nu_n$ with internal addresses $1-2-\ldots  -2^n-2^n+1$. They are not renormalizable, so by Lemma \ref{Lem:periodic_bound}, $h(\nu_n)>\log 2/ 2^n$. Since $\Diff(\nu_n,\nu)\ge 2^n+1$, we obtain:
\[
|h(\nu_n)-h(\nu)|> \frac{\log 2}{2^n} > \log 2 \frac{1}{\Diff(\nu_n,\nu)}.
\]
Hence, core entropy is not H\"older continuous at $\nu$. 
\end{proof}

\Newpage

\section{External angles and kneading sequences}
\label{Sec:angles}

The conjecture about H\"older continuity of core entropy is often formulated in terms of external angles. Douady and Hubbard showed that the Mandelbrot set $\M$ is compact, connected and full (i.e. $\C\setminus \M$ is connected), and there is a unique conformal isomorphism $\Phi: \Cbar \sm \M \to\Cbar\sm\diskbar$ 
with $\lim_{c \to \infty} \Phi(c)/c \to 1 $. Under the assumption that the Mandelbrot set is locally connected, $\Phi^{-1}$ extends to the boundaries as a continuous surjection. Hence, every point $c \in  \partial \M$ can be described by one (or several) \emph{external angles} $\theta\in \R/\Z =\Circle$ such that $\Phi^{-1}(e^{2\pi i\theta})=c$. For every $c$ there may be finitely many choices for $\theta$, but this does not change the dynamics described below: all such angles $\theta$ are known to have the same kneading sequence. The assumption of local connectivity of $\M$ is not really necessary: every $c\in\partial\M$ is still associated to at least one and at most finitely many external angles $\theta$ \cite{Fibers2,IntAddr}.

Every external angle $\theta\in\Circle$ has an associated kneading sequence defined as follows:
we divide the unit circle $\Circle$ at the points $\theta/2$ and $(\theta+1)/2$ (the preimages of $\theta$ under angle doubling); see Figure~\ref{Fig:KneadingAngle}. Let $A_\1$ be the open arc of $\Circle$ containing $\theta$ and let $A_\0$ be the other arc. The kneading sequence $\nu(\theta)=\nu_1\nu_2\ldots $ is defined as follows: for $k\ge 1$ and $i\in \{\0,\1\}$, if $2^{k-1}\theta \in A_i$, then $\nu_k=i$. If $2^{k-1}\theta \in \{\theta/2, (\theta+1)/2\}$, then $\nu_k=\star$; this implies that $\nu(\theta)$ is $\star$-periodic with period $k$. 

\begin{figure}[htbp]
\framebox{
\begin{picture}(220,220)
\includegraphics[width=.6\textwidth]{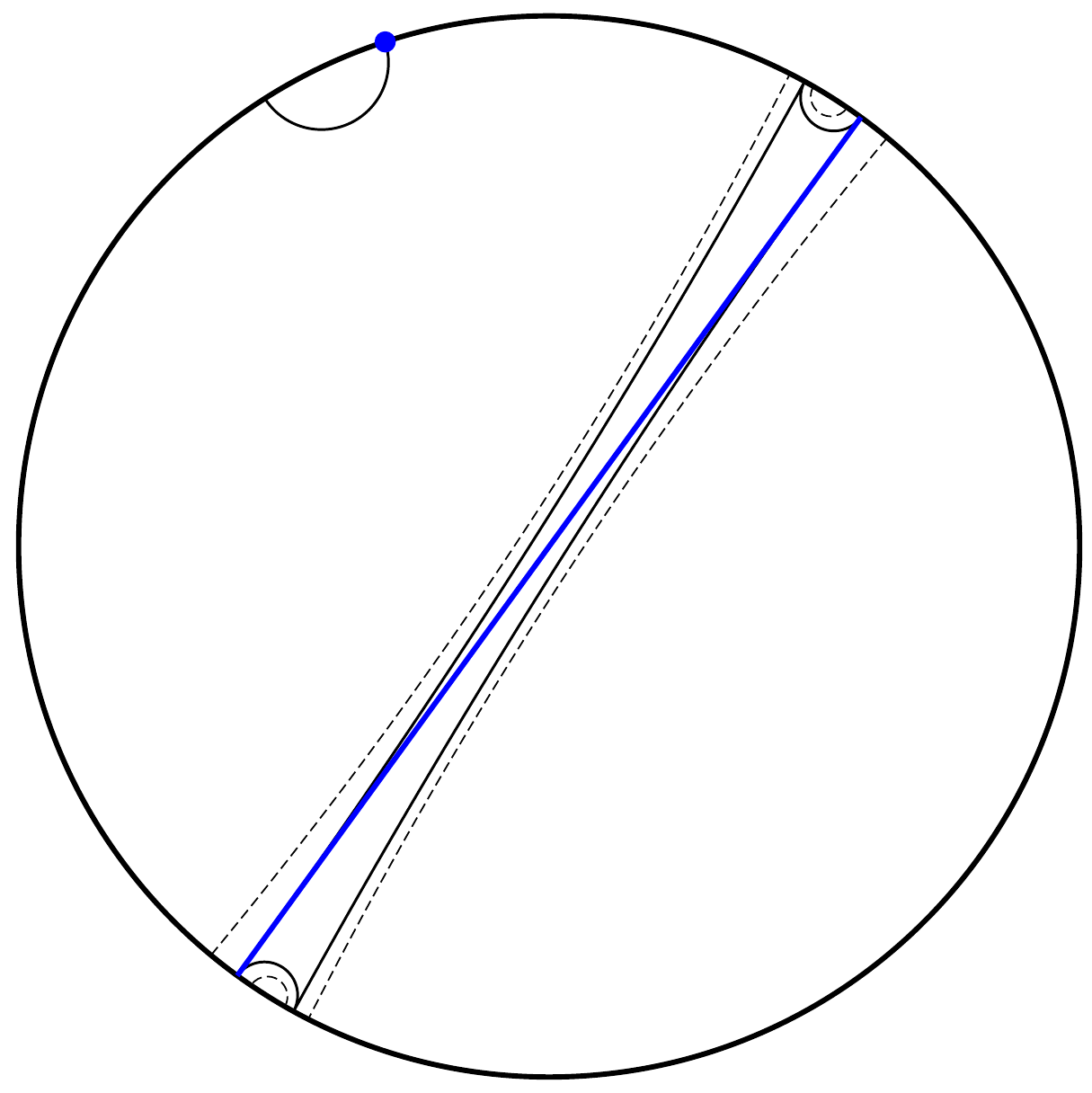}
\put(-145,212){$\theta$}
\put(-168,203){$\theta'$}
\put(-160,160){$A_{\1}$}
\put(-60,60){$A_{\0}$}
\end{picture}
}
\framebox{
\begin{picture}(220,140)(0,105)
\includegraphics[width=.6\textwidth]{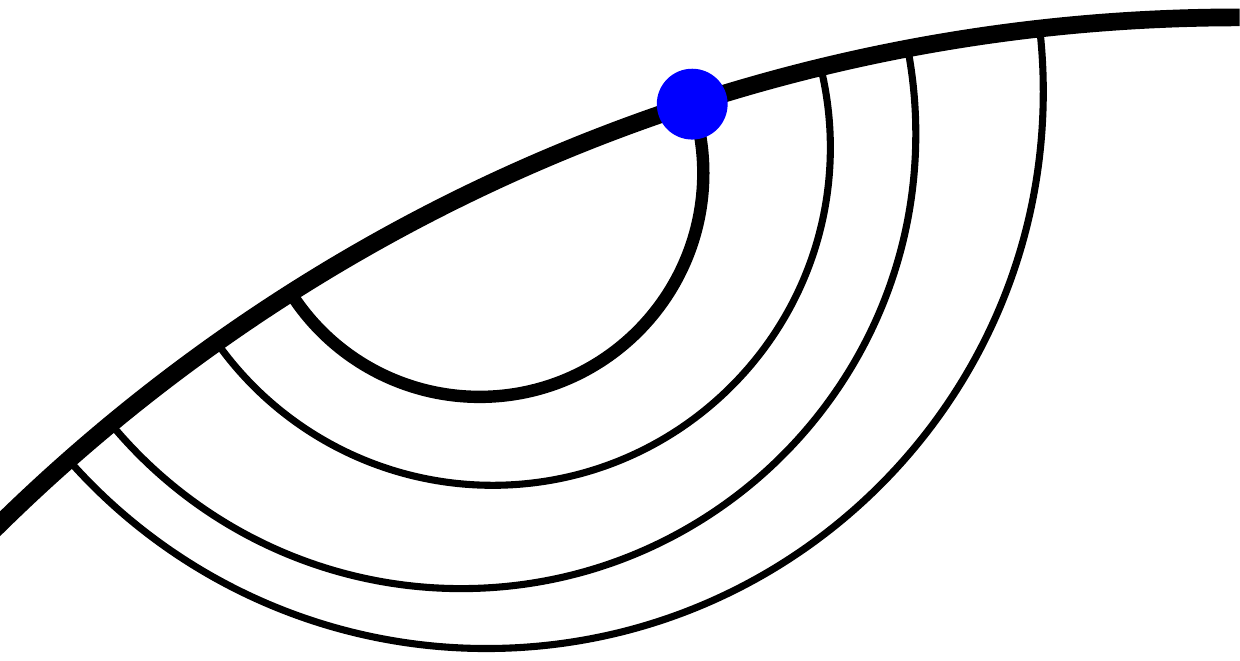}
\put(-105,222){$\theta$}
\put(-173,185){$\theta'$}
\put(-82,228){$\theta'_k$}
\put(-190,179){$\theta_k$}
\put(-70,231){$\theta_{n+k}$}
\put(-213,169){$\theta'_{n+k}$}
\put(-43,234){$\theta'_n$}
\put(-218,156){$\theta_n$}
\end{picture}
}
\caption{The kneading sequence as defined in terms of an external angle $\theta$ (marked by a blue dot). Top: the two preimages of $\theta$ form the precritical diameter (marked in blue); the two complementary domains (on $\Circle$ or $\disk$) are called $A_\0$ and $A_\1$. The companion angle $\theta'$ of $\theta$ is marked as well. The four preimages of $\theta$ and $\theta'$ form the central quadrilateral (drawn in solid lines); it is approximated by precritical leaves (representative leaves are indicated as dashed lines). Bottom: detail around the minor leaf with some approximating leaves as constructed in the proof.}
\label{Fig:KneadingAngle}
\end{figure}

\begin{lemma}[Recurrent kneading sequences]
\label{Lem:recurrent_angles}
For every angle $\theta\in\Circle$, the kneading sequence is a recurrent sequence in $\{\0,\1\}^\infty$ if and only if $\theta$ is recurrent under angle doubling on $\Circle$.
\end{lemma}
\begin{proof}
The statement excludes angles that are periodic under angle doubling: their associated kneading sequences are periodic sequences that involve the symbol $\star$, and we do not view them as recurrent.

If an angle $\theta$ is recurrent, then the associated kneading sequence is obviously recurrent as well. The converse is a bit more interesting. It is based on Thurston's theory of quadratic minor laminations \cite{Thurston}: this lamination extends the dynamics of angle doubling on the circle to parts of the disk bounded by the circle. 

Fix the angle $\theta$ of which we compute the kneading sequence. Then every angle $\phi\in\Circle$ has an associated \emph{itinerary}: this is the sequence $\tau_n\in\{\0,\1, \star\}^\infty$ so that the sequence $2^n\phi$ visits the arc $A_{\tau_n}$ at time $n$ (the $\star$ symbol occurs when one of the two preimages of $\theta$ is hit; this happens at most once except when $\theta$ is periodic). The 
itinerary of the partition angle $\theta$ itself is the kneading sequence). 

The problem is that the subset of angles  $\phi\in\Circle$ that generate the same initial $n$ entries in their itineraries is not connected. We obtain, however, a connected subset of the disk $\disk$ inside $\Circle$, as follows: the angle $\theta$ has preimages on $\Circle$ under angle doubling, $\phi/2$ and $(\phi+1)/2$. Connect these two by a diameter in $\disk$. The angle $\theta$ is associated to the critical value of a quadratic polynomial (in terms of a dynamic ray that lands or accumulates there); in the same sense, the diameter is associated to the critical point, which is the landing point of the two preimage rays. 

We continue to construct precritical points by induction: each of these is associated to a pair of external angles that are connected in $\disk$ (by a geodesic with respect to the Euclidean metric of $\disk$, or more customarily with respect to the hyperbolic metric). The two angles associated to each precritical point have a total of four preimages, and these can be connected in pairs within $\disk$ in exactly one way by a geodesic so as to not cross the connections generated by previously constructed precritical points; see for instance \cite{Thurston}. Each such geodesic is called a \emph{(precritical) leaf} in the lamination associated to the angle $\theta$ (in the language of Thurston, this lamination is not ``clean''). 

We define the \emph{central gap} as follows: we construct all precritical leaves of depths $2$ and more (that is, all except the diameter corresponding to the critical point); the complementary component in $\disk$ containing the origin (and thus the critical diameter) is called the central gap. Thurston \cite{Thurston} shows that the central gap is either a diameter or a quadrilateral, except if the kneading sequence periodic --- a case that is excluded for us. 

This implies that every angle on $\Circle$, except the two or four vertices of the central gap, are separated from the central gap by at least one precritical leaf. Let us first consider the case that the central gap is a quadrilateral (it has four vertices). Then for every $\eps>0$ there are four precritical leaves so that every angle in $\Circle$ outside the $\eps/2$-neighborhood of the four vertices of the central quadrilateral is separated from the central quadrilateral by one of these four precritical leaves. Taking images by angle doubling, the four vertices of the central quadrilateral map to $\theta$ and another angle, say $\theta'$ called the \emph{companion angle of $\theta$}. In this case 
there are two precritical leaves, say $\ell_\eps$ and $\ell'_\eps$, that separate $\theta$ and its companion angle from every angle outside their $\eps$-neighborhood. 

If the minimal depth of $\ell_\eps$ and $\ell'_\eps$ is $n$, then this implies that angles for which the itinerary coincides with that of $\theta$ for more than $n$ entries are $\eps$-close to either $\theta$ or its companion angle. 

Therefore, if the kneading sequence is recurrent, infinitely many angles $\theta_n=2^n\theta$ are $\eps$-close to $\theta$ or to $\theta'$. If infinitely many are $\eps$-close to $\theta$, then we are done (since $\eps>0$ was arbitrary). It remains to discuss the case that infinitely many $\theta_n$ are $\eps$-close to $\theta'$. 

First observe that the \emph{minor leaf} connecting $\theta$ to $\theta'$ is  
disjoint from all precritical leaves: if some precritical leaf were to intersect the minor leaf, then it would also intersect one of the precritical leaves that approximate the minor leaf (the image of a precritical close to the central gap), but precritical leaves are disjoint by construction. 

Let $\ell_n$ be the length of the image leaf $(2^n\theta,2^n\theta')$, for $n\ge 0$. Our next claim is that  $\ell_n\ge \ell_0$ for all $n$. If not, let $m$ be minimal such that $l_m<l_0$. Then either $\ell_{m-1}=\ell_m/2$ or $\ell_{m-1}=1-\ell_m/2$ by the standard formula for lengths of image leaves. The first case is excluded by minimality of $m$, and the second one implies that $\ell_{m-1}$ must intersect the central gap, which is impossible also. 

This implies that for $n\ge 1$, no image angle $\theta_n$ or $\theta'_n$ can be in $(\theta,\theta')$: otherwise, both $\theta_n$ and $\theta'_n$ must be in that interval, a contradition to the minimal length of the minor leaf.

By compactness, the angles $\theta_n'$ for which $\theta_n$ is close to $\theta'$ need to accumulate somewhere on the circle. Since the minor leaf is approximated by precritical leaves and has the shortest length among its forward orbit, the only possible accumulation point is $\theta$. 

Fix a particular $n$ for which the leaf $(\theta'_n,\theta_n)$ is $\eps$-close to $(\theta, \theta')$ in this order. A sufficiently small neighborhood, say $U$, of $(\theta,\theta')$ maps homeomorphically to a neighborhood of $(\theta'_n,\theta_n)$ under $n$ iterations. Another image leaf $(\theta'_k,\theta_k)$ must lie in $U$ between the two other leaves. In particular, if the vector from $\theta$ to $\theta'_k$ points clockwise, so will its homeomorphic image under $2^n$. Hence, the leaf $(\theta_{n+k}, \theta'_{n+k})$ lies between $(\theta'_n,\theta_n)$ and the minor leaf. In particular, it is $\eps$-close to the minor leaf and recurrence of $\theta$ follows. 

The other case is that the central gap is degenerate in the sense that it consists of a diameter alone, made up of the two preimage angles of $\theta$. The argument then simplifies as $\theta'=\theta$.\end{proof}

\begin{lemma}[H\"older continuity in terms of external angles]
Suppose an external angle $\theta$ is such that core entropy is H\"older continuous with exponent $h>0$ for the kneading sequence $\nu(\theta)$, locally in a neighborhood of $\nu(\theta)$. Suppose also that $\theta$ is non-recurrent with respect to angle doubling. Then core entropy as a function of angles is H\"older continuous at $\theta$ with exponent $(h-\eps)/\log 2$ for every $\eps>0$.
\label{Lem:angles}
\end{lemma}

\begin{proof}
By hypothesis $\nu=\nu(\theta)$ has an index $N>0$ so that if $k:=\diff(\nu,\mu)>N$, then $|h(\nu)-h(\mu)|<e^{-hk}$. 

Consider an angle $\phi$ such that $2^{-n-1}\le |\phi-\theta|<2^{-n}$ for $n>N$ and set $\mu:=\nu(\phi)$. Let $k:=\diff(\mu,\nu)$. 
Then there is a periodic angle $\theta_*$ between $\theta$ and $\phi$ of period $k$, and $|\theta-\theta_*|<2^{-n}$. It follows
\begin{align*}
 |2^k\theta-\theta| &\le |2^k(\theta-\theta_*)| + |2^k\theta_*-\theta_*|+|\theta_*-\theta| \le  (2^k+1) |\theta_*-\theta| \\&<(2^k+1)2^{-n} =2^{k-n}+2^{-n}
\;.
\end{align*}

Now fix $\eps>0$. If there are angles $\phi$ arbitrarily close to $\theta$ with associated numbers $n=n(\phi)$ and $k=k(\phi)$ as above such that $k(\phi)\le (1-\eps/h)n(\phi)$, then we have $k(\phi)-n(\phi)\le -(\eps/h) n(\phi)$ and thus
\[
|2^{k(\phi)}\theta-\theta| \le 2^{-(\eps/h) n(\phi)}+2^{-n(\phi)}
\;;
\]
since $\phi$ close to $\theta$ means $n(\phi)$ is large, this implies that $\theta$ is recurrent.

By hypothesis, $\theta$ is not recurrent, so there must be a neighborhood of $\theta$ in which all angles $\phi\neq\theta$ satisfy $k(\phi)\ge (1-\eps/h)n(\phi)$. In this neighborhood, we have
\begin{align*}
|h(\phi)-h(\theta)|&< e^{-hk} \le e^{-h(1-\eps/h)n} = e^{-(h-\eps)n}=\\
2^{-((h-\eps)/\log 2)n}&=  2^{(h-\eps)/\log 2} (2^{-n-1})^{(h-\eps)/\log 2} <  2 \cdot |\phi-\theta|^{(h-\eps)/\log 2}
\;. 
\end{align*}
\end{proof}

\begin{proof}[Proof of the Main Theorem]
By Lemma~\ref{Lem:recurrent_angles}, non-recurrent angles have non-recurrent kneading sequences. If the kneading sequence $\nu(\theta)$ (or equivalently the associated polynomial if it exists) is non-renormalizable and has a finite Hubbard tree, then by Theorem~\ref{Thm:Hoelder_continuity} core entropy as a function of kneading sequences is locally H\"older continuous with exponent $h(\nu)-\eps/2$.
If $\nu$ is renormalizable with a base sequence with positive entropy, then Theorem~\ref{Thm:HoelderRenorm} yields H\"older continuity for kneading sequences with exponent $h(\nu)$. If $\nu$ is renormalizable and tree-finite such that all base sequences have zero entropy, H\"older continuity for kneading sequences with exponent $h(\nu)-\eps/2$ follows from Theorem~\ref{Thm:Hoelder_low}. In each case, Lemma~\ref{Lem:angles} yields local H\"older continuity for angles with exponent $(h(\theta)-\eps)/\log 2$. 
\end{proof}

\end{document}